\documentclass[11pt]{amsart}

\usepackage{hyperref}
\hypersetup{
    colorlinks=true,
    citecolor=blue,
    filecolor=blue,
    linkcolor=blue,
    urlcolor=blue,
    linktoc=all
}

\usepackage[utf8]{inputenc}
\usepackage[margin=1in,marginparwidth=.7in]{geometry}
\usepackage[dvipsnames]{xcolor}
\usepackage{bbm}
\usepackage{tikz}
\usetikzlibrary{arrows}
\usepackage{tikz-cd}
\usepackage{pgffor}
\usepackage{amssymb,amsmath,amsfonts,mathrsfs,amsthm}
\usepackage{enumitem}
\usepackage{adjustbox}
\usepackage{hyperref}
\usepackage{relsize}
\usepackage{graphicx}
\usepackage[font={small}]{caption}
\usepackage{float}
\usepackage{subcaption}
\usepackage{changepage}
\usepackage{changes}
\usepackage[toc]{appendix}
\usepackage{pgfplots}
\pgfplotsset{compat=1.7}
\usepackage{adjustbox}
\usepackage{varwidth}
\usepackage{chngcntr}
\usepackage{standalone}
\usepackage{tcolorbox}
\usepackage{mathtools}
\usepackage{rotating}

\usepackage{tikz}
\usetikzlibrary{knots}
\usetikzlibrary{arrows,decorations.markings}

\definecolor{lightgray}{rgb}{.85,.85,.85}

\newtheorem{theorem}{Theorem}[section]
\newtheorem{lemma}[theorem]{Lemma}
\newtheorem{proposition}[theorem]{Proposition}
\newtheorem{corollary}[theorem]{Corollary}

\newtheorem{algorithm}[theorem]{Algorithm}

\newtheorem{question}[theorem]{Question}

\theoremstyle{remark}
\newtheorem{example}[theorem]{Example}

\theoremstyle{remark}
\newtheorem{remark}[theorem]{Remark}

\theoremstyle{definition}
\newtheorem{definition}[theorem]{Definition}

\theoremstyle{remark}
\newtheorem{rmk/}{Remark}

\newcommand{\R}{\mathbb{R}}
\newcommand{\Q}{\mathbb{Q}}
\newcommand{\Z}{\mathbb{Z}}
\newcommand{\C}{\mathbb{C}}

\newcommand{\Hom}{\operatorname{Hom}}
\newcommand{\Ext}{\operatorname{Ext}}

\newcommand{\im}{\operatorname{im}}

\renewcommand{\l}{\ell}

\newcommand{\gr}{\operatorname{gr}}

\newcommand{\BN}{\mathcal{BN}}

\newcommand{\F}{\mathcal{F}}
\newcommand{\til}[1]{\widetilde{#1}}
\renewcommand{\b}[1]{\overline{#1}}

\newcommand{\lr}[1]{\vert {#1} \vert}

\newcommand{\rar}[1]{\xrightarrow{#1}}

\newcommand{\D}{\mathcal{D}}

\renewcommand{\o}{\otimes}

\newcommand{\s}{\mathfrak{s}} 

\renewcommand{\:}{\colon}

\newcommand{\bF}{\mathbb{F}}
\newcommand{\la}{\lambda}
\newcommand{\G}{\mathcal{G}}

\newcommand{\inv}{^{-1}}
\newcommand{\st}{\ | \ }

\newcommand{\resp}{resp.\ }

\newcommand{\CKh}{\mathit{CKh}} 
\newcommand{\CKhD}{\mathit{CKh}_{\D}} 
\newcommand{\CFK}{\mathit{CFK}}
\newcommand{\HFK}{\mathit{HFK}}
\newcommand{\CKhred}{\widetilde{\CKh}} 
\newcommand{\Khred}{\widetilde{\Kh}}

\newcommand{\essl}{\mathfrak{sl}}

\newcommand{\grq}{\gr_q}
\newcommand{\grt}{\gr_t}
\newcommand{\grtred}{\widetilde{\gr}_t} 

\newcommand{\Kh}{\mathit{Kh}} 
\newcommand{\KhD}{\Kh_{\D}}
\newcommand{\Kg}{\mathit{Kg}} 
\newcommand{\CKhredD}{\CKhred_{\D}}
\newcommand{\KhredD}{\Khred_{\D}}

\newcommand{\sat}{s_t}

\newcommand{\bp}{\mathbf{p}} 

\newcommand{\colora}{\mathbf{a}} 
\newcommand{\colorb}{\mathbf{b}} 
\newcommand{\obar}{\bar{o}} 

\renewcommand{\r}{\mathfrak{r}}

\newcommand{\onto}{\twoheadrightarrow} 

\title[Concordance invariants from $U(1) \times U(1)$-equivariant Khovanov homology]{Concordance invariants from $U(1) \times U(1)$-equivariant Khovanov homology}

\author[R. Akhmechet]{Rostislav Akhmechet}
\address{Department of Mathematics, Columbia University, New York NY 10027}
\email{\href{mailto:akhmechet@math.columbia.edu}{akhmechet@math.columbia.edu}}

\author[M. Zhang]{Melissa Zhang}
\address{Department of Mathematics, UC Davis, One Shields Ave., Davis, CA 95616-8633, U.S.A.
}
\email{\href{mailto:mlzhang@ucdavis.edu}{mlzhang@ucdavis.edu}}

\begin{document}

\maketitle

\begin{abstract}
    We study Khovanov homology over the Frobenius algebra $\bF[U,V,X]/((X-U)(X-V))$, or $U(1) \times U(1)$-equivariant Khovanov homology, and extract two families of concordance invariants using the algebraic $U$-power and $V$-power filtrations on the chain complex. We also further develop the reduced version of the theory and study its behavior under mirroring. 
\end{abstract}
\tableofcontents

\section{Introduction}

Khovanov homology ($\Kh$), or $\essl(2)$ link homology, was introduced by Khovanov in \cite{Kh1} and has since been studied in and applied to many contexts. In particular, numerical invariants such as Rasmussen's $s$ invariant \cite{Rasmussen}, distilled from homology-type invariants of knots (e.g.\ Khovanov-Lee homology) have proven remarkably useful in low-dimensional topology. The first prominent use of $s$ was Rasmussen's purely combinatorial reproof of the Milnor conjecture on the slice genus of torus knots, which was previous proven by Kronheimer and Mrowka using gauge theory \cite{KM}. 
More recently, Piccirillo \cite{Piccirillo} used $s$ to show that the Conway knot is not slice, answering a famous, long-standing question. 

Heegaard Floer homology is another homology-type invariant in low-dimensional topology, introduced by Ozsv\'ath and Szab\'o in \cite{OS-3mfld} as a 3-manifold invariant. Knot Floer homology ($\HFK$), defined by Ozsv\'ath and Szab\'o \cite{O-S_HFK} and independently by Rasmussen \cite{Rasmussen_HFK}, is an invariant for knots in 3-manifolds defined by adding a filtration to the 3-manifold's Heegaard Floer complex. 

Khovanov homology has its roots in representation theory, whereas Heegaard Floer homology is built on geometric foundations; nevertheless, the algebraic structure of their complexes have many similarities. For example, $\Kh$ and $\HFK$ are both defined as chain complexes with distinguished generators, and both complexes are bigraded. As a result, many constructions have counterparts on either side. For example, Rasmussen's distillation of $s$ from $\Kh$ is akin to Ozsv\'ath and Szab\'o's distillation of the concordance invariant $\tau$ from $\HFK$ \cite{O-S_tau}: both invariants measure the internal grading of a standard nontorsion portion of the homology.

Nowadays, $\Kh$ and $\HFK$ come in many different flavors, depending on the algebraic properties of the base ring. In this article we will focus in particular on $U(1) \times U(1)$-equivariant Khovanov homology, previously studied by Sano in \cite{Sano-divisibility,Sano-fixing-functoriality} and Khovanov-Robert in \cite{KRfrobext}. This flavor is an extension of the so-called `universal" $U(2)$-equivariant Khovanov homology  (denoted $\F_5$ in \cite{Khfrobext}). The underlying ground ring $R$ and Frobenius algebra $A$ are given by 
\[
 R = \bF[U,V], \hskip2em A = R[X]/((X-U)(X-V))
\]
where $\bF$ is a field. Similarly, the minus flavor of knot Floer homology over the ground ring $\mathbb{F}_2[U,V]$ can also be viewed as the most general flavor. 

For a knot $K \subset S^3$ and $t\in[0,2]$, Ozsv\'ath, Stipsicz, and Szab\'o defined a concordance homomorphism $\Upsilon$ (from the smooth knot concordance group to the group of piecewise-linear functions $[0,2] \to \R$)
using the minus complex $\CFK^-(K, S^3)$ \cite{OSS-upsilon}. Livingston provided a useful visual reinterpretation of $\Upsilon$ \cite{livingston-upsilon}, which we will also refer to throughout the text. The invariant $\Upsilon$ recovers $\tau$ (as the derivative at $t=2$) and has been used to identify structures in the smooth knot concordance group \cite{O-S_tau, OSS-upsilon, Hom_HFK_and_smooth_concordance, DHST}. 
The goal of this article is to define a counterpart $\sat$ to $\Upsilon$ coming from Khovanov homology and study what types of information $\sat$ may illuminate. There are indeed already many $\Upsilon$-like invariants coming from Khovanov homology: 
\begin{itemize}
    \item Lewark-Lobb's gimel invariant $\gimel$ \cite{Lewark-Lobb}, from $\essl(n)$ knot homology defined via $\C[x]/(x^n - x^{n-1})$-modules; this invariant is ``interesting" for quasi-alternating knots, which are ``non-interesting" from the point of view of Khovanov homology and knot Floer homology \cite{Manolescu-Ozsvath-qalt}.
    \item Grigsby-Licata-Wehrli's \cite{GLW} annular Rasmussen invariants $d_t$, from annular Khovanov-Lee homology; 
    \item Sarkar-Seed-Szab\'o's generalized Rasmussen invariants $s^\mathcal{U}$ from their perturbation on Szab\'o's geometric spectral sequence \cite{Sarkar-Seed-Szabo};  
    \item Truong-Zhang's annular Sarkar-Seed-Szab\'o invariants $s_{r,t}$, from annular Sarkar-Seed-Szab\'o homology \cite{Truong-Zhang}; 
    \item and Ballinger's $\Phi_K(\alpha)$, from using both Lee's perturbation as well as Rasmussen's $E(-1)$ differential on Khovanov homology \cite{Ballinger}.
\end{itemize}
All of the above invariants come from mixing two fundamentally different filtrations on the chain complex; for example, $\gimel$ mixes the quantum filtration with the $x$-power filtration on the differential. 

The motivation for the present work is to explore the information that can be gleaned from mixing the $U$-power and $V$-power filtrations on $U(1) \times U(1)$-equivariant Khovanov homology; this construction is in some ways more analogous to the original definition of $\Upsilon$. 
This is part of a larger research program to discover the topological information that can mined from working with $U(1) \times U(1)$-equivariant Khovanov homology, which has a larger ground ring than ``universal" $U(2)$-equivariant Khovanov homology. This enlargement provides additional flexibility, shown to be useful in recent developments: these include another method of {fixing the sign ambiguity in} Khovanov's original theory  \cite{Sano-fixing-functoriality}, and {defining an equivariant analogues of annular link homology} \cite{Akh, AKh-Kh}.

\subsection{New Invariants}

We offer two families of $\Upsilon$-like Rasmussen invariants for knots in $S^3$.
The first, $s_t$, is extracted from $U(1) \times U(1)$-equivariant Khovanov homology (denoted $\Kh(K)$), by measuring the complex using a filtration grading that mixes the quantum, $U$-power, and $V$-power gradings. 

\begin{theorem}
For a link $L$, there is a function $s_t(L)
: [0,2] \to \R$ which is piecewise-linear continuous. Moreover, if $K$ and $K'$ are concordant knots then $s_t(K) = s_t(K')$ for all $0\leq t \leq 2$. 
\end{theorem}

The above theorem is proven as Theorem \ref{thm:concordance-invariant} and
Theorem \ref{thm:continuity} below. The second concordance invariant, $\til{s}_t$, is defined analogously to $s_t$ but using reduced version of $U(1)\times U(1)$-equivariant Khovanov homology (denoted $\Khred(K)$), which we define in Section \ref{sec:reduced}. This version is more closely analogous to knot Floer homology since the free part of the homology is rank 1 (``there is one tower"). 

Independently from our work, Sano and Sato are also studying a reduced version of $U(1) \times U(1)$-equivariant Khovanov homology in an upcoming article \cite{Sano-Sato-upcoming}.

\subsection{Acknowledgements} We thank Jen Hom, Mikhail Khovanov, Gage Martin, and {Taketo Sano} for useful discussions. 
The first author was partially supported by NSF Grant DMS-2105467, NSF RTG Grant DMS-1839968, and the Jefferson Scholars Foundation during the writing of this article. This article was completed while the second author was supported by the National Science Foundation under Grant No.~DMS-1928930, while she was in residence at the Simons Laufer Mathematical Sciences Institute (previously known as MSRI) in Berkeley, California, during the Fall 2022 semester. 

\section{Background}

\subsection{Filtrations}
\label{sec:filtrations}

We begin with the algebraic background on the types of filtrations on chain complexes featured in the later sections. Let $\bF$ be a field and let $M$ be an $\bF$-vector space. 
\begin{definition}
\label{def:filtration}
A \emph{(descending) filtration} of $M$ consists of a family $F = \{F_\la\}_{\la \in \R}$ of $\bF$-subspaces of $M$, satisfying 
\begin{enumerate}
    \item $F_\la \subset F_{\la'}$ whenever $\la \geq \la'$. 
    \item $\bigcup_{\la \in \R} F_\la = M$. 
    \item There exists $\mu\in \R$ such that  $F_\la = \{0\}$ for all $\la > \mu$. \label{item:filt is zero}
\end{enumerate}
\end{definition}

Given a filtration $F$, we can define a function $\gr_F \: M\to \R \cup \{\infty\}$, called the \emph{grading induced by $F$}, by setting 
\begin{equation}
    \label{eq:gr_F}
    \gr_F(m) = \sup \{ \la  \mid m\in F_\la \}.
\end{equation}
Note that $\gr_F(m) = \infty$ if and only if $m=0$. 
We also define
\begin{equation}
\label{eq:gr_F(M)}
    \gr_F(M) = \sup \{ \gr_F(m) \mid m\in M\setminus \{0\}\}.
\end{equation}

The filtrations we consider will often satisfy the following additional properties. 

\begin{definition}
\label{def:filtration admits a grading and discrete}
Consider a filtration $F = \{F_\la\}_{\la\in \R}$ of $M$.  
\begin{enumerate}[label= (F\arabic*)]
    \item We say $F$ \emph{admits a grading} if for all nonzero $m\in M$, there exists $\la(m) \in \R$ such that $m\in F_{\la(m)}$ and $m\not\in F_{\la}$ for any $\la > \la(m)$.  \label{item:F1}
    \item We say $F$ is \emph{discrete} if each $F_\la$ is finite-dimensional. \label{item:F2}
\end{enumerate}
\end{definition}

 As an example of a filtration which does not satisfy \ref{item:F1}, consider $M= \bF$ and set $F_\la = \bF$ if $\la<0$, $F_\la = \{0\}$ if $\la\geq 0$.

\begin{remark}
Property \ref{item:F1} is equivalent to saying that for all nonzero $m\in M$, $\gr_F(m)$ is actually a maximum rather than a supremum; in particular, $\gr_F(m) = \la(m)$; compare this with \cite[Definition 5]{GLW}. The term \emph{discrete} in \cite[Definition 3]{GLW} and \cite[Definition 3.1]{livingston-upsilon} means that $F_{\la'}/F_{\la}$ is finite dimensional for all $\la \geq \la'$. All of our filtrations will satisfy item \eqref{item:filt is zero} of Definition \ref{def:filtration}, so this notion of discreteness is equivalent to property \ref{item:F2}. 
\end{remark}

\subsubsection{Filtered vector spaces with distinguished generators}
\label{sec:filtered-distinguished-generators}

Before continuing, we pause to describe which kinds of vector spaces and filtrations will be considered in the following sections. The underlying vector space is the $U(1)\times U(1)$-equivariant Khovanov chain complex $\CKh(D)$, recalled in Section \ref{sec:link homology}. It is a chain complex of free $\bF[U,V]$-modules with a finite distinguished $\bF[U,V]$-basis denoted $\Kg_0(D)$ (the \emph{Khovanov generators}; cf.\ Section \ref{sec:kh-gradings}). 
Then as an $\bF$-vector space, $\CKh(D)$ is free with basis $Kg(D) = \{U^m V^n g \mid g\in \Kg_0(D),\ m,n\geq 0\}$. 
For each $t\in [0,2]$, we will define a filtration $F^t$ on $\CKh(D)$ which is preserved by the differential and which admits a grading. Moreover, for $t\in (0,2)$, we show that $F^t$ is discrete; see Lemma \ref{lem:grt defines a filtration}. The values $t=0,2$ will be treated separately. 

\subsubsection{Extracting numerical invariants}
\label{sec:Extracting numerical invariants}
We are ultimately interested in extracting numerical invariants from homology, so we discuss induced filtration gradings on subspaces and quotients. To that end, let $F$ be a filtration on $M$, and let $N\leq M$ be a subspace. Then $N$ inherits a filtration $F'$ by setting $F'_\la = F_\la \cap N$.  For quotients, let $\b{F}$ denote the induced filtration on $M/N$, defined by setting $\b{F}_\la$ to consist of all cosets $m+N \in M/N$ which have a representative in $F_\la$. Equivalently, $\b{F}_\la$ is the image of $F_\la$ under the projection $M\to M/N$. Define the induced grading $\gr_{\b{F}}$ as in \eqref{eq:gr_F}.

\begin{lemma}
\label{lem:filtrations} 
Suppose the filtration $F$ of $M$ is discrete and admits a grading. 
\begin{enumerate}
    \item The filtrations $F'$ of $N$ and $\b{F}$ of $M/N$, for any subspace $ N \leq M$, are discrete and admit gradings.
    \item If $M\neq \{0\}$,  $\gr_F(M) = \gr_F(m)$ for some nonzero $m\in M$.
    \item If $m+N \in M/N$ is nonzero, there exists $m'\in M$ with $\gr_{\b{F}}(m+N) = \gr_F(m')$. 
    \item If $N\subsetneq M$, then there exists $m\in M\setminus N$ such that $\gr_{\b{F}}(M/N) = \gr_F(m)$. 
\end{enumerate}
\end{lemma}
\begin{proof}
For item $(1)$, we show that $M/N$ admits a grading; the other statements are clear. Let $m+N \in M/N$ be nonzero, and set $\la = \gr_{\b{F}}(m+N)$. It suffices to show that $m+N \in \b{F}_\la$. Suppose for the sake of contradiction that $m+N  \not\in \b{F}_\la$. Fix a representative $m_0$ of $m+N$, and set $\la_0 = \gr_F(m_0)$. Since $F$ admits a grading, we know $m_0 \in F_{\la_0}$, so $\la > \la_0$. Then there exists a representative $m_1$ of $m+N$ with $\la > \la_1 := \gr_F(m_1) > \la_0$. Then the inclusion $F_{\la_1} \subset F_{\la_0}$ is proper. Continuing, we obtain an infinite descending sequence of proper subspaces of $F_{\la_0}$, contradicting \ref{item:F2}.

Item $(2)$ is similar. Suppose $\gr_F(M) > \gr_F(m)$ for all $m\in M\setminus\{0\}$. Pick $m_0\in M\setminus \{0\}$, and set $\la_0 = \gr_F(m_0)$. Then there exists $m_1\in M\setminus\{0\}$ such that $\la_1:= \gr_F(m_1) > \la_0$. Property \ref{item:F1} implies that $m_0 \in F_{\la_0} \setminus F_{\la_1}$. Continuing, we obtain an infinite descending sequence of proper subspaces of $F_{\la_0}$, contradicting \ref{item:F2}.  

For item $(3)$, let $\la = \gr_{\b{F}}(m+N)$. Since $M/N$ admits a grading, we know that $m+N \in \b{F}_\la$, so $m+N$ has a representative $m'$ with $m'\in F_\la$, which implies $\gr_F(m') = \la$.

Item $(4)$ follows from the first three. 
\end{proof}

\subsection{Background on Khovanov homology}

This section gives an overview of the construction of Khovanov homology from a formal complex over the Bar-Natan cobordism category, introduced in \cite{BN2}. Given a planar link diagram $D$ of a link $L$, Bar-Natan defined a chain complex $[[D]]$ over an appropriate category, and showed the chain homotopy type of $[[D]]$ is an invariant of $L$. Applying a suitable $(1+1)$-dimensional TQFT to $[[D]]$ yields homological invariants of oriented links.

\subsubsection{The Bar-Natan category and formal chain complex}\label{sec:The formal complex}

We start with a brief overview of the Bar-Natan category $\BN$ and the formal chain complex associated to oriented link diagrams, as introduced in \cite{BN2}.
Let $I :=[0,1]$ denote the unit interval, and fix a field $\bF$.
Objects of $\BN$ are defined to be formal direct sums of formally graded collections of disjoint simple closed curves in the plane $\R^2$.  
Morphisms are matrices whose entries are formal $\bF$-linear combinations of cobordisms properly embedded in $\R^2 \times I$, modulo isotopy relative to the boundary, and subject to the local relations shown in Figure \ref{fig:BN relations}. 
The degree of a cobordism $S$ is defined to be
\begin{equation}\label{eq:deg of cob}
\deg(S) = \chi(S).
\end{equation}
Note that the Bar-Natan relations are homogeneous.
 
\begin{remark}
\label{rem:why-field}
The Bar-Natan category and the formal complex can be defined over any commutative ring. We choose to work over a field in order to streamline the algebra needed for a well-behaved filtrations. For example, the proof of Lemma \ref{lem:filtrations}  relies on the fact that finitely generated $\bF$-modules (i.e.\ finite-dimensional vector spaces) are Artinian.
\end{remark}

\subsubsection{To Khovanov homology}
We will now pass from Bar-Natan's construction to Khovanov homology. The formal grading of objects and degrees of morphisms in $\BN$ will correspond to the \emph{quantum grading} $\gr_q$ in Khovanov homology.

\begin{remark}
Throughout this article, we use the term ``grading" to describe elements of a graded ring or graded module. We reserve the term ``degree" for the change in grading associated to maps.
\end{remark}

\begin{figure}
\centering
\begin{subfigure}[b]{.3\textwidth}
\begin{center}
\begin{tikzpicture}
\draw[thick] (0,0) circle [radius=20pt];
\draw (-.7,0) to[bend right] (.7,0);
\draw[densely dashed] (-.7,0) to[bend left] (.7,0);
\node[draw=none, anchor = west] at (.7,0) {$= 0$};
\end{tikzpicture}
\end{center}
\caption{Sphere}\label{fig:S}
\end{subfigure}
\begin{subfigure}[b]{.4\textwidth}
\begin{center}
 \begin{tikzpicture}[thick]
  \draw (-.5,0) to[bend left] (.5,0);
  \draw (-.7,.1) to[bend right] (.7,.1);
  \draw[rotate=0] (0,0) ellipse (40pt and 20pt);
  \node[draw=none, anchor = west] at (1.4,0) {$= 2$};
\end{tikzpicture}
\end{center}
\caption{Torus}\label{fig:T}
\end{subfigure}\\
\vskip1em
\begin{subfigure}[b]{1 \textwidth}
\begin{center}
\begin{tikzpicture}[thick]

\node[draw=none] at (0,0) {$=$};
\node[draw=none] at (-3,0) {$+$};
\node[draw=none] at (3,0) {$+$};
\draw (-1.2,-.8) arc (180:360:10pt and 5pt);
\draw[densely dashed] (-.5,-.8) arc (0:180:10pt and 5pt);

\draw (-1.2,.8) arc (180:360:10pt and 5pt);
\draw(-.5,.8) arc (0:180:10pt and 5pt);

\draw[thick] (-1.2,-.8) -- (-1.2, .8);
\draw[thick] (-.5,-.8) -- (-.5, .8);

\draw (-2.5,-.8) arc (180:360:10pt and 5pt);
\draw[densely dashed] (-1.8,-.8) arc (0:180:10pt and 5pt);

\draw (-2.5,.8) arc (180:360:10pt and 5pt);

\draw (-1.8,.8) arc (0:180:10pt and 5pt);

\draw (-1.8,-.8) arc (0:180:10pt and 15pt);
\draw (-2.5,.8) arc (180:360:10pt and 15pt);


\draw (1.8,-.8) arc (180:360:10pt and 5pt);
\draw[densely dashed] (2.5,-.8) arc (0:180:10pt and 5pt);

\draw (1.8,.8) arc (180:360:10pt and 5pt);
\draw(2.5,.8) arc (0:180:10pt and 5pt);

\draw (.5,-.8) arc (180:360:10pt and 5pt);
\draw[densely dashed] (1.2,-.8) arc (0:180:10pt and 5pt);

\draw (.5,.8) arc (180:360:10pt and 5pt);
\draw (1.2,.8) arc (0:180:10pt and 5pt);


\draw (1.2,.8) arc (180:360:8.5pt and 8pt);
\draw (.5,.8) arc (180:360:28.5pt and 25pt); 

\draw (1.2,-.8) arc (0:180:10pt and 15pt);
\draw (2.5,-.8) arc (0:180:10pt and 15pt); 


\begin{scope}[shift={(-3-3,0)},yscale=1,xscale=-1] 
\draw (-1.2,-.8) arc (180:360:10pt and 5pt);
\draw[densely dashed] (-.5,-.8) arc (0:180:10pt and 5pt);

\draw (-1.2,.8) arc (180:360:10pt and 5pt);
\draw(-.5,.8) arc (0:180:10pt and 5pt);

\draw[thick] (-1.2,-.8) -- (-1.2, .8);
\draw[thick] (-.5,-.8) -- (-.5, .8);

\draw (-2.5,-.8) arc (180:360:10pt and 5pt);
\draw[densely dashed] (-1.8,-.8) arc (0:180:10pt and 5pt);

\draw (-2.5,.8) arc (180:360:10pt and 5pt);

\draw (-1.8,.8) arc (0:180:10pt and 5pt);

\draw (-1.8,-.8) arc (0:180:10pt and 15pt);
\draw (-2.5,.8) arc (180:360:10pt and 15pt);
\end{scope}

\begin{scope}[shift={(3,0)},yscale=-1,xscale=1]
\draw (1.8,-.8) arc (180:360:10pt and 5pt);
\draw (2.5,-.8) arc (0:180:10pt and 5pt);

\draw[densely dashed] (1.8,.8) arc (180:360:10pt and 5pt);
\draw(2.5,.8) arc (0:180:10pt and 5pt);

\draw (.5,-.8) arc (180:360:10pt and 5pt);
\draw(1.2,-.8) arc (0:180:10pt and 5pt);

\draw[densely dashed] (.5,.8) arc (180:360:10pt and 5pt);
\draw (1.2,.8) arc (0:180:10pt and 5pt);


\draw (1.2,.8) arc (180:360:8.5pt and 8pt);
\draw (.5,.8) arc (180:360:28.5pt and 25pt); 

\draw (1.2,-.8) arc (0:180:10pt and 15pt);
\draw (2.5,-.8) arc (0:180:10pt and 15pt); 
\end{scope}
\end{tikzpicture}
\end{center}
\caption{Four tubes}\label{fig:4Tu}
\end{subfigure}
\caption{Relations in $\BN$.}\label{fig:BN relations}
\end{figure}

Let $D$ be a diagram for an oriented link $(L,o)$, where $o$ is an orientation on the link $L \subset \R^3$. We describe the formal complex $[[D]]$ from \cite{BN2}. First, form the so-called \emph{cube of resolutions} as follows. Label the crossings of $D$ by $1,\ldots, n$. There are two ways to smooth a crossing, called the \emph{0-smoothing} and \emph{1-smoothing}, shown in \eqref{eq:two smoothings}. For each sequence $u = (u_1,\ldots, u_n)\in \{0,1\}^n$, perform the $u_i$-smoothing at the $i$-th crossing. This yields a collection of disjoint simple closed curves in the plane, which we denote by $D_u$. Elements of $\{0,1\}^n$ are vertices of an $n$-dimensional cube, and we decorate the vertex $u$ by the planar diagram $D_u$. 

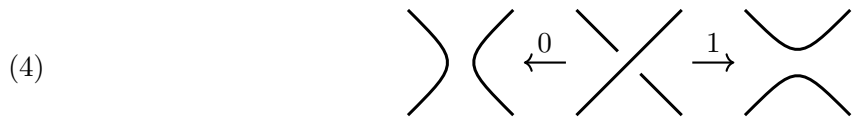
\begin{equation}\label{eq:two smoothings}
\begin{aligned}
\begin{tikzpicture}[scale=.7]
\draw[line width=.4mm] (0,2) -- (2,0);

\draw[line width=.4mm, preaction={draw, line width=12pt, white}] (0,0) -- (2,2); 

\draw[->, thick] (-.2,1) -- (-1,1);
\draw[->, thick] (2.2,1) -- (3,1);

\node[draw=none, anchor = south] at (-.6,1) {$0$};
\node[draw=none, anchor = south] at (2.6,1) {$1$};

\draw[line width=.4mm] (-1.2,0) .. controls (-2.2,1) .. (-1.2,2);
\draw[line width=.4mm] (-3.2,0) .. controls (-2.2,1) .. (-3.2,2); 

\draw[line width=.4mm] (3.2,0) .. controls (4.2,1) .. (5.2,0);
\draw[line width=.4mm] (3.2,2) .. controls (4.2,1) .. (5.2,2); 

\end{tikzpicture}
\end{aligned}
\end{equation}

Let $u=(u_1,\ldots, u_n)$ and $v=(v_1,\ldots, v_n)$ be vertices which differ only in the $i$-th entry, in which $u_i = 0$ and $v_i=1$. In other words, there is an edge from $u$ to $v$ in the $n$-dimensional cube. The two smoothings $D_u$ and $D_v$ are identical outside of a small disk around the $i$-th crossing. There is a cobordism from $D_u$ to $D_v$ consisting of the obvious saddle ($1$-handle attachment) near the $i$-th crossing, and the identity (product cobordism) outside of the small disk. Denote this cobordism by $d_{u,v}$, and decorate each edge of the $n$-dimensional cube by the corresponding saddle cobordism. We now have a commutative cube in the category $\BN$. There is a way to associate a sign $s_{u,v} \in \{\pm 1\}$ to each edge in the cube so that multiplying the edge map $d_{u,v}$ by $s_{u,v}$ results in an anti-commutative cube (see \cite[Section 2.7]{BN2} for details). 

For a vertex $u=(u_1,\ldots, u_n) \in \{0,1\}^n$, set $\lr{u} = \sum_i u_i$. The chain complex $[[D]]$ is defined by setting 
\begin{equation}
\label{eq:cpx-q-shift}
[[D]]^i = \bigoplus_{\vert u \vert = i+ n_-} D_{u}\{i + n_+ -n_-\} 
\end{equation}
where $n_-$, $n_+$ are the number of negative and positive crossings in $D$ respectively, and the brackets $\{ \cdot \}$ denote the formal grading shift in $\BN$. 
The differential $d$ is given on each direct summand by the edge map $s_{u,v}d_{u,v}$. Anti-commutativity of the cube implies that $d^2=0$. 
\begin{theorem}\emph{(\cite[Theorem 1]{BN2})}\label{thm:[[D]] is invariant}
If diagrams $D$ and $D'$ are related by a Reidemeister move, then $[[D]]$ and $[[D']]$ are chain homotopy equivalent.
\end{theorem}

The formal complex $[[D]]$ is, in a sense, a universal invariant for $\essl(2)$ homological invariants. Applying a functor from $\BN$ into an abelian category yields link homology, and Theorem \ref{thm:[[D]] is invariant} ensures that the homotopy class of the resulting chain complex is a link invariant. Constructing such a functor amounts to finding a $(1+1)$-dimensional TQFT which factors through the Bar-Natan relations in Figure \ref{fig:BN relations}; this is the content of the next section.

\subsection{$U(1) \times U(1)$-equivariant Khovanov homology} 
\label{sec:link homology}
We now describe the Frobenius pair $(R,A)$ and its associated TQFT. The ground ring is given by 
\[
R= \bF[U,V],
\]
where $\bF$ is our fixed field in grading zero, with gradings of $U$ and $V$ given by $\gr(U) = \gr(V) = -2$. The Frobenius algebra is 
\[
A= R[X]/((X-U)(X-V))
\]
with the trace $\varepsilon \: A \to R$ given by 
\[
\varepsilon(1) = 0, \ \varepsilon(X)=1.
\]
Comultiplication is then
\begin{align*}
\Delta \:  A & \to A \otimes A\\
 1 & \mapsto X \o 1 + 1 \o X - (U+V) 1\o 1\\
 X & \mapsto X\otimes X - U V 1\otimes 1 .
\end{align*}
As noted in \cite{KRfrobext}, $R$ and $A$ are the $U(1) \times U(1)$-equivariant cohomology with $\bF$ coefficients of a point and $2$-sphere $S^2$, respectively.

\begin{remark}
\label{rem:U(2) pair}
The pair $(R,A)$ is an extension of the so-called $U(2)$-equivariant Frobenius pair $(R^{\rm sym}, A^{\rm sym})$, with $R^{\rm sym} = \bF[E_1, E_2]$, $A^{\rm sym} = R^{\rm sym}[X]/(X^2-E_1X+E_2)$. The inclusion 
\[
(R^{\rm sym}, A^{\rm sym}) \hookrightarrow (R,A)
\]
is given by identifying $E_1, E_2$ with the elementary symmetric polynomials in $U$ and $V$: 
\[
E_1 \mapsto U+V, \qquad  E_2 \mapsto UV.
\]
The pair $(R^{\rm sym}, A^{\rm sym})$ is denoted $\F_5$ in \cite{Khfrobext}, and the link homology obtained from it is sometimes called \emph{universal} or \emph{equivariant} Khovanov homology.
\end{remark}

The TQFT $\F$ factors through the Bar-Natan relations. For a planar link diagram $D$, denote by
\[
\CKh(D)
\]
the chain complex of graded $R$-modules obtained by applying $\F$ to the Bar-Natan complex $[[D]]$. Theorem \ref{thm:[[D]] is invariant} implies that the chain homotopy class of $\CKh(D)$ is invariant under Reidemeister moves, and we let $\Kh(D)$ denote its homology. 

\begin{remark}
Setting $U=V=0$ recovers the Frobenius algebra $A_0 = \bF[X]/(X^2)$ and Khovanov's complex from \cite[Section 7]{Kh1}, which we will denote $\CKh_0(D)$. Setting $(U,V) = (\pm 1, \mp 1)$ collapses the grading into a filtration and yields Lee's complex \cite{Lee}. Setting one of $U,V$ to zero yields a graded version of Bar-Natan's complex (denoted $\F_6$ in \cite{Khfrobext}). 
\end{remark}

\subsubsection{Gradings}
\label{sec:kh-gradings}

Note that $A$ is a free $R$-module with basis $\{1,X\}$. Define a grading on $A$ by 
\begin{equation}
\label{eq:gradings}
\gr(1) = 1 \hskip2em \gr(X) = -1.
\end{equation} 
Multiplication by $U$ or $V$ lowers the above grading by $2$. Let $\F$ denote the corresponding $(1+1)$-dimensional TQFT. A cobordism $\Sigma \subset \R^2 \times I$ from $Z_0$ to $Z_1$ induces a map $\F(\Sigma) \: A^{\o n_0} \to A^{\o n_1}$ of degree $\chi(S)$, where $n_i$ is the number of circles in $Z_i$.

\begin{remark}
\label{rem:gradings}
There are two grading conventions appearing in the literature. The gradings in \eqref{eq:gradings} agree with \cite{Kh1, BN1, BN2}. Other sources \cite{Khfrobext, KRfrobext} use the opposite grading where the ground ring is non-negatively graded, while in the present paper it is non-positively graded.
Moreover, when viewing $A$ as an $R$-algebra, it is more natural to set $\gr(1) = 0$, $\gr(X) = -2$, so that multiplication preserves the grading. However, in Khovanov homology, one views $A$ as an $R$-module, and gradings are balanced around zero as above.
\end{remark}

Suppose a link diagram $D$ has $n_+$ positive crossings and $n_-$ negative crossings, totaling $n = n_+ + n_-$ crossings.
The chain complex $\CKh(D)$
has a distinguished $R$-basis given at each smoothing $\F(D_u)$ by elements 
\[
g=g_1 \o \dots \o g_k
\]
where $k$ is the number of circles in $D_u$ and each $g_i \in \{1, X\}$. We call such a basis element a \emph{Khovanov generator},
and denote the set of Khovanov generators by $\Kg_0(D)$ (cf.\ Section \ref{sec:filtered-distinguished-generators}).
The quantum grading of $g\in \Kg_0(D)$ is
\begin{equation}
\label{eq:grq chain complex}
\grq(g) = \gr(g) + \lr{u} + n_+ - 2n_-
\end{equation}
where $\gr(g)$ is induced by the formula in equation \eqref{eq:gradings}.  We will also view $\CKh(D)$ as an $\bF$-vector space, in which case a distinguished basis is given by elements of the form $U^{m} V^{n} g$, where $g\in \Kg_0(D)$ is a Khovanov generator as above, and $m, n \geq 0$. Denote this $\bF$-basis by $\Kg(D)$. We have 
\begin{equation*}
\grq(U^{m} V^{n} g) = \grq(g) - 2(m + n). 
\end{equation*}

Note that the quantum grading $\gr_q$ is distinct from the homological grading of the complex, which we denote by $\gr_h$.

\subsubsection{Maps associated to cobordisms}

We recall from \cite{Kh1} maps assigned to cobordisms. An oriented link cobordism $S \subset \R^3 \times I$ from an oriented link $L$ to an oriented link $L'$ can be represented as a sequence of movie frames $D_0, D_1,  \ldots, D_k$, where each $D_i$ is an oriented link diagram, $D_0$ and $D_k$ are diagrams for $L$ and $L'$ respectively, and $D_{i+1}$ is obtained from $D_i$ by either a Reidemeister move or a Morse move (cup, cap, or saddle). A chain map $\CKh(D_i) \to \CKh(D_{i+1})$ is naturally associated to each such elementary cobordism $D_i \to D_{i+1}$: to Reidemeister moves assign the corresponding chain homotopy equivalence, and to Morse moves assign the map induced by the TQFT $\F$ on each vertex in the cube of resolutions. Composing these elementary cobordism maps  yields a chain map $\CKh(D_0) \to \CKh(D_1)$ of degree $\chi(S)$, such that the induced map on homology is independent of the decomposition of $S$ into elementary pieces up to an overall sign \cite{Jacobsson, BN2, Kh-functoriality}.

\subsection{Inverting the discriminant}\label{sec:Inverting the discriminant}

We review a further extension $(R_\D, A_D)$ of the Frobenius pair $(R, A)$, introduced in \cite{KRfrobext}. Let 
\begin{equation}\label{eq:D}
\D = (V-U)^2
\end{equation}
denote the discriminant of the quadratic polynomial $(X-U)(X-V) \in R[X]$. Let 
\[
R_\D = R [ \D^{-1}]
\]
denote the ring obtained by inverting $\D$ (equivalently, one may invert $V - U$), and let
\[
A_\D = A \o_{R} R_\D
\]
be the extension of $A$ to an $R_\D$-algebra. Let $\F_\D = \F\o_R R_\D$ denote the corresponding TQFT. For a link $L\subset \R^3$ with diagram $D$, let 
\[
\CKhD(D) := \F_\D([[D]])
\]
be the resulting chain complex. It is an invariant of $L$ up to chain homotopy equivalence, and we will denote its homology by $\KhD(L)$.

The elements 
\begin{equation}\label{eq:e_1 and e_2}
e_1  = \frac{X-U}{V-U},\hskip1em e_2 = -\frac{X - V}{V-U} \in A_\D. 
\end{equation}
form a basis for $A_\D$ and satisfy 
\[
e_1 + e_2 = 1,\ e_1^2 = e_1,\ e_2^2 = e_2,\ e_1e_2=0,
\]
so that the algebra $A_\D$ decomposes as a product, $A_\D = R_\D e_1 \times R_\D e_2$. With respect to the basis $\{e_1, e_2\}$, comultiplication in $A_\D$ is simply given by
\begin{equation}\label{eq:comultiplication D}
\begin{split}
\Delta(e_1) &= (V- U) e_1 \o e_1 \\
\Delta(e_2) &= - (V-U) e_2\o e_2.
\end{split}
\end{equation}

\subsubsection{Involutions}

As explained in \cite[Section 1.1]{KRfrobext}, there is an $R$-algebra involution $ \iota: A \to A$ given by
\begin{equation}
\label{eq:iota involution}
   \iota(1) = 1, \quad \iota(X) = U+ V - X.
\end{equation}
Note that $\iota$ interchanges $X-U$ and $V-X$. We obtain an involution $\iota^{\otimes k}$ on $A^{\otimes k}$ by applying $\iota$ on each tensor factor. 

Although $\iota$ is an algebra involution, it does not preserve the Frobenius algebra structure, since $\Delta(\iota(a)) = - \Delta(a)$ and  $\varepsilon(\iota(a)) = -\varepsilon(a)$ for all $a\in A$. That is, the comultiplication and counit maps are twisted (in the sense of \cite{Khfrobext}) by $-1$. Consequently, applying $\iota^{\otimes k(u)}$ at each vertex $u$ of the cube of resolutions (where $k(u)$ is the number of planar circles in the resolution $D_u$) does not yield a chain map. 
However, by \cite[Proposition 3]{Khfrobext} (see also the discussion of Sch\"utz in \cite[Section 2]{Schuetz-integral}) one can make sign adjustments at the vertices of the cube of resolutions. That is, applying $\sigma(u)\cdot \iota^{\otimes k(u)}$ for some sign $\sigma(u) \in \{\pm 1\}$ at each vertex $u$ yields an honest involution on the chain complex $\CKh(D)$. We denote this involution by $I$.  

\begin{remark}
One can, for example, choose $\sigma(u)$ as follows. Pick any path $\gamma$ in the cube of resolutions from the all-$0$ resolution to $u$, and let $\# \Delta(u)$ denote the number of comultiplications (split maps) along the path $\gamma$. Note that $\# \Delta(u)$ does not depend on the choice of the path $\gamma$. Define $\sigma(u)$ to be $(-1)^{\# \Delta(u)}$.  
\end{remark}

Furthermore, $\iota$ induces an involution on $A_\D$, which swaps $e_1$ and $e_2$, and by introducing signs as above we obtain an involution of the localized complex $\CKhD(D)$ which agrees with the chain map induced by the chain map $I : \CKh(D) \to \CKh(D)$. By an abuse of notation, we will use $\iota$ to denote the involution on $A$ or $A_\D$, and $I$ to denote the involution on both chain complexes $\CKh(D)$ and $\CKhD(D)$.

\subsubsection{Properties of the localized theory}

As noted in \cite[Section 1.2]{KRfrobext}, the TQFT $\F_\D$ is essentially the Lee deformation \cite{Lee}. By \cite[Theorem 4.2]{Lee}, the Lee homology of an $l$-component link is free (over $\Q$) of rank $2^l$. An alternate proof can be found in the final remark in \cite{Wehrli}; see also \cite{BN-Morrison}. The following proposition is readily proven using the same arguments. 

\begin{proposition}\label{prop:inverting D rank}
For a link $L\subset \R^3$ with $l$ components, the homology $\KhD(L)$ is a free $R_\D$-module of rank $2^l$.
\end{proposition}

Let \[
T(D) = \{\eta \in \Kh(D) \mid r\cdot\eta = 0 \text{ for some nonzero } r\in R\}.
\]
denote the $R$-torsion submodule of $\Kh(D)$. 
Proposition \ref{prop:inverting D rank} implies the following; see also \cite[Corollary 2.10]{Sano-divisibility}. 

\begin{corollary}
\label{cor:only c torsion}
The (non-localized) homology $\Kh(D)$ contains only $(V-U)$-torsion. That is, every element of $T(D)$ is annihilated by a power of $V-U$. 
\end{corollary}

\begin{proof}
Exactness of localization at $V-U$ yields an inclusion $T(D) \o_{R} R_{\D} \hookrightarrow \Kh(D) \o_{R} R_{\D} \cong \KhD(D)$. But $\KhD(D)$ is free (over the integral domain $R_\D$) by Proposition \ref{prop:inverting D rank}, while every element in $T(D) \o_{R} R_{\D}$ is torsion. Therefore $T(D) \o_{R} R_{\D} = 0$, which completes the proof. 
\end{proof}

\subsubsection{Canonical Lee generators for the localized theory}
Following \cite[Section 4.4]{Lee} (see also \cite[Section 2.4]{Rasmussen}, \cite[Algorithm 2.5]{Sano-fixing-functoriality}), an explicit basis for $\KhD(L)$ consisting of the \emph{canonical Lee classes} of $\KhD(D)$ can be constructed as follows. (Note that these are different from the Khovanov generators, which are distinguished generators for the non-localized complex.)

\begin{algorithm}
\label{alg:ABcolors}
Suppose we are given an orientation $o$ of $D$ and the checkerboard coloring $C$ (shaded/unshaded) of the complement of the diagram in $\R^2$ such that the unbounded region is unshaded. For each circle in the oriented resolution $D_o$, draw a small dot slightly to the left of the circle, based on its orientation. If the dot is in a shaded region, label the circle $\colora$; if the dot is in an unshaded region, label the circle $\colorb$. 
\end{algorithm}

In the direct summand of $\CKhD(D)$ corresponding to $D_o$, consider the element $\s_o$ whose tensor factor is $e_1$ for $\colora$-labeled circles and $e_2$ for $\colorb$-labeled circles. The element $\s_o$ is a cycle, called a \emph{canonical Lee generator}, whose homology class $[\s_o]$ is a canonical Lee class of the homology of the localized complex.
See Figure \ref{fig:trefoil-example} for an example. The homology classes $[\s_o]$ over all orientations of $D$ form a basis for $\KhD(D)$. Given an orientation $o$, the reversed orientation is denoted $\b{o}$, and the associated canonical Lee generator is denoted $\s_{\b{o}}$ accordingly. 
The involution $I$ swaps $\s_o$ and $\s_{\b{o}}$, up to a sign. 

\begin{figure}
    \centering
\raisebox{-1.7cm}{
\begin{tikzpicture}
\def\l{2}
\begin{knot}[
consider self intersections,
flip crossing=1,
flip crossing=3,
clip width=10,
]
\strand[thick,black]
(90:2) to [out=180,in=-120,looseness=\l]
(-30:2) to [out=60,in=120,looseness=\l]
(210:2) to [out=-60,in=0,looseness=\l] 
(90:2);
\end{knot}
\draw [decoration={markings, mark=at position 1 with {\arrow[scale=2,>=stealth]{>}}}, postaction={decorate}] (0,2) -- (-0.1,2);
\draw [decoration={markings, mark=at position 1 with {\arrow[scale=2,>=stealth]{>}}}, postaction={decorate}] (0,.75) -- (-0.1,.75);
\foreach \i in {0,120,240}
{
    \rotatebox{\i}{
    \fill[lightgray] (0,1.9) arc(90:-10:1) arc(70:110:2.9) arc(190:90:1) -- (0,1.9);
    }
}
\end{tikzpicture}
}

\hspace{-1cm}

\def\s{5} 
\def\t{0.5*\s} 

\begin{tikzpicture}
\draw[thick,fill=lightgray,rounded corners=15mm] (0,0) -- (\s,0) -- (.5*\s,\s*.866)  -- cycle;
\begin{scope}[transform canvas={xshift=2.5cm}]
    \rotatebox{60}{
        \draw[thick,fill=white,rounded corners=10mm] (0,0) -- (\t,0) -- (.5*\t,\t*.866)  -- cycle;
    }
\end{scope}
\draw[decoration={markings, mark=at position 1 with {\arrow[scale=2,>=stealth]{>}}}, postaction={decorate}] (2.5,3.56) -- (2.4,3.56);
\begin{scope}[transform canvas={yshift=-1.4cm}]
    \draw[decoration={markings, mark=at position 1 with {\arrow[scale=2,>=stealth]{>}}}, postaction={decorate}] (2.5,3.56) -- (2.4,3.56);
\end{scope}
\draw[black,fill=black] (2.5,3.35) circle (.5ex);
\draw[black,fill=black] (2.4,1.95) circle (.5ex);
\node[draw=none, anchor = west] at (2.5,1.9) {$\colora$};
\node[draw=none, anchor = west] at (2.5,3.3) {$\colorb$};
\end{tikzpicture}
    \caption{An oriented diagram $(D,o)$ of a right-handed trefoil and the associated canonical Lee generator $\s_o$. The shading indicates the choice of checkerboard coloring.}
    \label{fig:trefoil-example}
\end{figure}
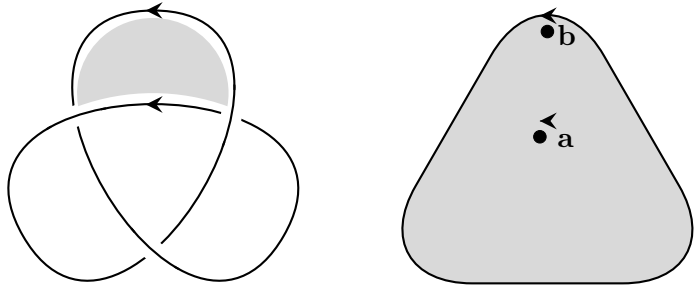

\subsubsection{Maps associated to cobordisms}

The following is an analogue of \cite[Proposition 4.1]{Rasmussen}, proven by Sano  for $U(1) \times U(1)$-equivariant Khovanov homology:

\begin{proposition}\emph{(\cite[Proposition 3.17]{Sano-divisibility})}\label{prop:localized-map-injective}
Let $(L,o), (L',o')$ be oriented links and $S$ a connected cobordism from $L$ to $L'$. Fix oriented diagrams $(D,o)$ and $(D',o')$ for $L$ and $L'$ respectively. Then, up to a sign, the induced map
\[
S_* \: \KhD(D) \to \KhD(D')
\]
sends $\s_o$ and $\s_{\b{o}}$ to $(V-U)$-multiples of $\s_{o'}$ and $\s_{\b{o}'}$, respectively. In particular, $S_*$ is an isomorphism if $L$ and $L'$ are knots. 

\end{proposition}

\begin{corollary}
\label{cor:induced map is injective mod torsion}
Let $K, K'$ be oriented knots and $S$ a connected cobordism from $K$ to $K'$. Fix oriented diagrams $(D,o)$ and $(D',o')$ for $K$ and $K'$ respectively. Then the induced map
\[
\til{S}_* \: \Kh(D)/T(D) \to \Kh(D')/T(D')
\]
is injective.
\end{corollary}

\begin{proof}
Consider the commutative square 
\[
\begin{tikzcd}
    \Kh(D) \ar[r,"{S}_*"] \ar[d, "\lambda"'] & \Kh(D') \ar[d, "\lambda'"] \\
    \KhD(D) \ar[r, "\cong"] & \KhD(D')
\end{tikzcd}
\]
where the horizontal maps are induced by $S$, $\lambda, \lambda'$ are the natural localization maps, and the bottom horizontal map is an isomorphism by Proposition \ref{prop:localized-map-injective}. The kernels of $\lambda$ and $\lambda'$ are the submodules of $\Kh(D)$ and $\Kh(D')$ consisting of $(V-U)$-torsion elements, which by Corollary \ref{cor:only c torsion} are equal to $T(D)$ and $T(D')$, respectively. Thus we have a commutative square 
\[
\begin{tikzcd}
    \Kh(D)/T(D)\ar[r,"\til{S}_*"] \ar[d, hook] & \Kh(D')/T(D') \ar[d, hook] \\
    \KhD(D) \ar[r, "\cong"] & \KhD(D')
\end{tikzcd}
\]
with injective vertical maps, which guarantees injectivity of $\til{S}_*$. 
\end{proof}

\begin{remark}
For a knot $K$, $\KhD(K)$ is supported in homological grading zero. Hence the quotient $\Kh(K) /T(K)$ is also supported in homological grading zero. 
\end{remark}

\section{Concordance invariant from $U$ and $V$ algebraic gradings} 
\label{sec:construction}

In this section, we construct an Upsilon-like Rasmussen concordance invariant from $\CKh(D)$. We begin by combining the quantum grading with two algebraic filtration gradings, and study the behavior of the differential under these ``mixed" filtration gradings.

\subsection{A mixed filtration grading}

\subsubsection{$U$- and $V$-power algebraic gradings}
The complex $\CKh(D)$ comes with two additional algebraic gradings, $\gr_U$ and $\gr_V$ given by the $U$- and $V$-powers, respectively. Precisely, for an element $U^{m} V^{n} g\in \Kg(D)$, define 
\begin{equation}
\label{eq:gru grv}
\gr_U (U^{m} V^{n} g) = m, \ \ \gr_V(U^{m} V^{n} g) = n. 
\end{equation}
The differential is non-decreasing with respect to each of these gradings. 

On the $UV$-plane, we see that the differential $d$ has four types of homogeneous components, 
\[
    d = d_{0,0} + d_{1,0} + d_{0,1} + d_{1,1},
\]
where the subscripts denote $(\gr_U, \gr_V)$-bidegree. The first component, $d_{0,0}$, is the original Khovanov differential for $\Kh_0$. 

\begin{example}
Consider a one-crossing negative diagram for the unknot, so that the differential is just comultiplication. Then 
\begin{align*}
  d(1) &= \underbrace{X\o 1 + 1\o X}_{d_{0,0}} - \underbrace{U 1\o 1}_{d_{1,0}} - \underbrace{V 1\o 1 }_{d_{0,1}}, \\
  d(X) &= \underbrace{X \o X}_{d_{0,0}} - \underbrace{UV 1\o 1}_{d_{1,1}}.
\end{align*}
\end{example}

\begin{remark}
Compare this picture to Grigsby-Licata-Wehrli's \cite{GLW} annular Khovanov-Lee differentials on the $(j, j-2k)$-bigrading plane, where $j$ and $k$ refer to the quantum and annular gradings, respectively.
\end{remark}

\subsubsection{A mixed filtration grading $\grt$}

For $t \in [0,2]$, define a filtration grading $\gr_t$ that is homogeneous on the distinguished generators $\Kg(D)$ as follows:
\begin{equation}
\label{eq:grt grading}
    \gr_t( U^{m} V^{n} y) = \gr_q(y) - ( t \cdot m + (2-t) \cdot n).
\end{equation}
Multiplication by $U$ decreases $\gr_t$ by $t$, while multiplication by $V$ decreases $\gr_t$ by $2-t$.

In order for $\gr_t$ to be a filtration grading on $\CKh(D)$, we must first understand how $\gr_t$ changes under cobordisms. This will also allow us to prove some properties of the invariant in Section \ref{sec:properties}.

\begin{lemma}
\label{lem:grt filtered degree}
Let $Z_0$ and $Z_1$ be two collections of disjoint simple closed curves in $\R^2$, and let $S\subset \R^2 \times I$ be a cobordism from $Z_0 \subset \R^2\times\{0\}$ to $Z_1 \subset \R^2 \times \{1\}$. If $z \in \F(Z_0)$ is $\gr_t$-homogeneous, then every homogeneous summand $z'$ of $\F(S)(z) \in \F(Z_1)$ satisfies 
\[
\gr_t(z') \geq \gr_t(z) + \chi(S) .
\]
\end{lemma}

\begin{proof}
This is easily verified for the elementary cup, cap, and saddle cobordisms, from which the statement follows. 
\end{proof}

\begin{corollary}
\label{cor:d is non decreasing}
For a link diagram $D$, the differential $d$ in $\CKh(D)$ is non-decreasing with respect to $\gr_t$. 
\end{corollary}

\subsubsection{Filtration by $\grt$}

\begin{definition}
\label{def:F^t filtration}
For each $t\in [0,2]$ and $\la\in \R$, let $F_\la^t$ denote the $\bF$-span of all $\gr_t$-homogeneous elements $z\in Kg(D)$ with $\gr_t(z) \geq \lambda$. 
\end{definition}

Note that, for a $\gr_t$-homogeneous chain $z\in \CKh(D)$, the induced grading $\gr_{F^t}(z)$, as defined in \eqref{eq:gr_F}, is equal to $\gr_t(z)$. More generally, for any nonzero $z\in \CKh(D)$, the induced filtration grading $\gr_{F^t}(z)$
is equal to the minimum $\gr_t$-grading among all homogeneous summands of $z$ when $z$ is expressed as an $\bF$-linear combination of elements of $Kg(D)$. In a slight abuse of notation, we will write $\gr_t(z):= \gr_{F^t}(z)$ even for non-homogeneous $z$. 

Recall the two properties in Definition \ref{def:filtration admits a grading and discrete}. 

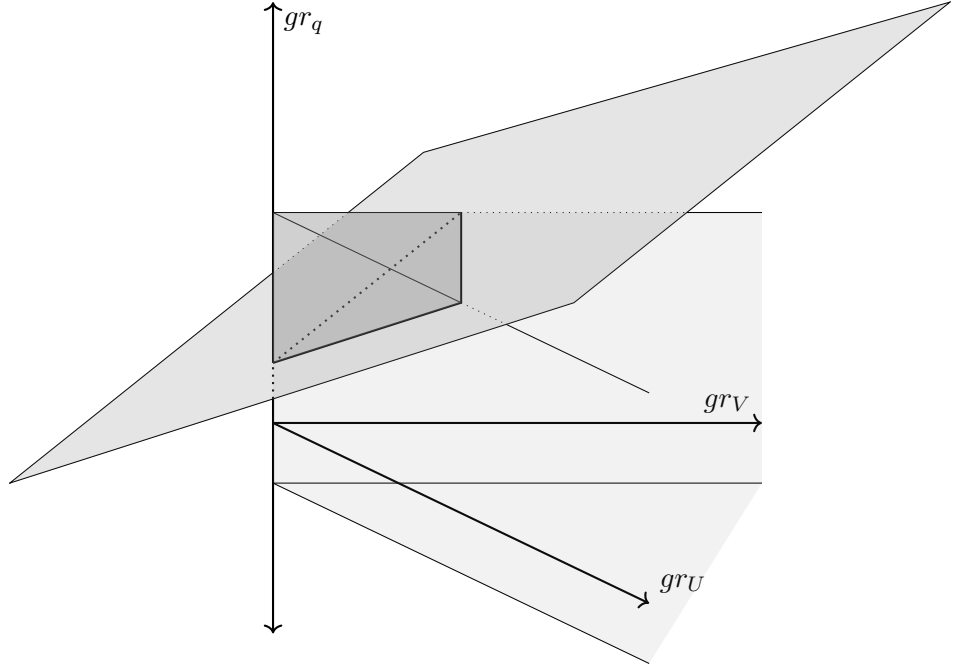
\begin{figure}
    \centering

\begin{tikzpicture}[xscale=0.5, yscale=.4]
\draw[->, thick] (0,3) -- (0,15) 
    node[anchor= north west] {$gr_q$};
\draw[dotted, thick] (0,1.8) -- (0,3);
\draw[->, thick] (0,1.8) -- (0,-6);
\draw[black] (0,8) -- (5,8);
\draw[black, dotted] (5,8) -- (11, 8);
\draw[black] (11, 8) -- (13,8);
\draw[black] (0,8) -- (5,5);
\draw[black, dotted] (5,5) -- (6.2, 4.28);
\draw[black] (6.2, 4.28) -- (10,2);
\draw[black] (0,6) --(-7, -1) -- (8,5) -- (18, 15) -- (4,10) -- (2,8);
\draw[black, dotted] (0,6) -- (2,8);
\draw[black, thick] (0,3) -- (5,5) -- (5,8);
\draw[black, thick, dotted] (0,3) -- (5,8);
\draw[black] (0,-1) -- (13,-1);
\draw[black] (0,-1) -- (10, -7);
\draw[black, ->, thick] (0,1) -- (13,1)
    node[anchor=south east] {$gr_V$};
\draw[black, ->, thick] (0,1) -- (10, -5)
    node[anchor=south west] {$gr_U$};
\fill[gray, opacity=0.2] (0,6) --(-7, -1) -- (8,5) -- (18, 15) -- (4,10) -- (0,6);
\fill[gray, opacity=0.1] (0,-1) -- (10,-7) -- (13,-1) -- (13, 8) -- (0,8) -- (0,-1);
\fill[gray, opacity=0.3] (0,3) -- (5,5) -- (5,8) -- (0,8);
\end{tikzpicture}
    \caption{An illustration of the proof of discreteness in Lemma \ref{lem:grt defines a filtration}. The image of $\kappa$ is discrete and contained in a set whose $\gr_q$ coordinate is bounded (the light pie-slice region). The plane consists of points whose dot product with $(-t, -(2-t), 1)$ equals $\la$. The darker conic region contains the images under $\kappa$ of all generators $U^m V^n y$ such that $\gr_t(U^m V^n y) \geq \la$. The plane is not parallel to the $\gr_U$ or $\gr_V$ axes when $t\in (0,2)$, implying that the conic region is compact, thus containing finitely many images of $\kappa$. Since each fiber of $\kappa$ is finite, this implies $F^t_\la$ is finitely generated. 
    }
    \label{fig:filtration is discrete}
\end{figure}

\begin{lemma}
\label{lem:grt defines a filtration}
For $t\in [0,2]$,  each $F_\la^t$ is an $\bF$-subcomplex of $\CKh(D)$, and $F^t = \{F_\la^t\}$ is a filtration (in the sense of Definition \ref{def:filtration}) which satisfies property \ref{item:F1}. If moreover $t\in (0,2)$, then $F^t$ satisfies \ref{item:F2}.  
\end{lemma}
\begin{proof}
That $F_\la^t$ is an $\bF$-subcomplex follows from Corollary \ref{cor:d is non decreasing}. Items (1) and (2) of Definition \ref{def:filtration} are immediate. For item \eqref{item:filt is zero}, let $\mu$ denote the maximal $\gr_q$-grading among all elements of $\Kg_0(D)$ (note that $\gr_q(z) = \gr_t(z)$ for $z\in \Kg_0(D)$). Since multiplication by $U$ and $V$ does not increase $\gr_t$, we have $F_\mu^t \neq \{0\}$ but $F_\la^t = \{0\}$ for all $\la>\mu$, verifying item \eqref{item:filt is zero}. To show that property \ref{item:F1} holds, take $z\in \CKh(D)$ and write $z$ as an $\bF$-linear combination of elements of $\Kg(D)$. Denote by $\la$ the minimal $\gr_t$-grading of these homogeneous summands. Then $z\in F_\la^t$ but $z\not\in F_{\la'}^t$ for any $\la'>\la$, which demonstrates property \ref{item:F1}.

Now we show that $F^t$ is discrete when $t\in (0,2)$. Consider the function 
\[
\kappa \: \Kg(D) \to \Z^3 
\]
given by $\kappa(U^m V^n g) = (m,n,\gr_q(g))$. The following properties are immediate:
\begin{itemize}
    \item $\gr_t(z)$ is equal to the dot product $\kappa(z) \cdot (-t,-(2-t), 1)$ for any $z\in \Kg(D)$. 
    \item Each fiber of $\kappa$ is finite. 
    \item The image of $\kappa$ is contained in $\{ (a,b,c) \in \R^3 \mid a,b\geq0, \vert c \vert \leq M\}$ for some $M\in \R$. 
\end{itemize}
By the first point,  $F^t_\la$ is the $\bF$-span of the preimage of all $(a,b,c) \in \Z^3$ such that $(a,b,c)\cdot (-t,-(2-t), 1) \geq \la$. When $t\in (0,2)$, the second and third points imply that this spanning set is finite, verifying property \ref{item:F2}. See Figure \ref{fig:filtration is discrete} for an illustration.
\end{proof}

By Lemma \ref{lem:filtrations}, if $t\in (0,2)$ then for any subspace $N\leq \CKh(D)$, the induced filtration grading $\gr_{\b{F}^t}$ on the quotient is given by the $\gr_t$-grading of some representative.
If $t=0,2$ then the filtration does not satisfy property \ref{item:F2}: indeed, if $z\in F^0_\la$, then $U^mz\in F^0_\la$ for any $m\geq 0$, and likewise if $z\in F^2_{\la}$ then $V^nz\in F^2_\la$ for any $n\geq 0$. However, we have the following analogous statement. 

\begin{lemma}
\label{lem:grading at 0,2}
Let $N\leq \CKh(D)$ be an $\bF$-subspace and let $t\in \{0,2\}$. Then for every $z\in \CKh(D)\setminus N$, $\gr_{\b{F}^t}(z+N) = \gr_F(z+n)$ for some  $n\in N$. Moreover, $\gr_{F^t}(\CKh(D)/N)$, as defined in \eqref{eq:gr_F(M)}, is equal to $\gr_{\b{F}^t}(z+N)$ for some $z\in \CKh(D)\setminus N$.  
\end{lemma}

\begin{proof}
At $t=0,2$, the set of possible gradings $\{\gr_{F^t}(z)\mid z\in \CKh(D)\setminus \{0\} \}$ is a discrete subset of $\R$ (in fact, it is contained in $\Z$). The statement follows. 
\end{proof}

Note that, while each $F^t_\la$ is preserved by the differential, it is not preserved under the action of the ground ring: $U^{m} V^{n} \cdot F_\la^t \subset F_{\la - (tm + (2-t) n)}^t$. 

\subsection{Definition of the concordance invariant $s_t$}

Fix a link diagram $D$ and $t\in [0,2]$. Consider the filtration $F^t$ on $\CKh(D)$ from Definition \ref{def:F^t filtration}. As discussed in Section \ref{sec:filtrations}, the homology $Kh(D)$ inherits a filtration from $F^t$. In what follows, we consider only the nontorsion classes, so we take the quotient of $\Kh(D)$ by the $R$-torsion submodule $T(D)$. The quotient $\Kh(D) /T(D)$ again inherits a filtration, as explained in Section \ref{sec:Extracting numerical invariants}. To simplify the notation we will denote the induced filtration grading on $Kh(D)/T(D)$ again by $\gr_t$. We are now ready to define our invariant. 

\begin{definition}
For a diagram $D$ of an oriented link $L$, let 
\begin{align}
\begin{aligned}\label{eq:def1}
\gr_t(D) &= \sup \{ \gr_t(\b{\eta}) \mid \b{\eta} \in \Kh(D)/T(D) \text{ is nonzero} \}, \\
s_t(D) &= \gr_t(D) - 1. 
\end{aligned}
\end{align}
We define $s_t(L):= s_t(D)$, where $D$ is any diagram of $L$. 
\end{definition}

Lemma \ref{lem:filtrations}, Lemma \ref{lem:grt defines a filtration}, and Lemma \ref{lem:grading at 0,2} imply that $\gr_t(D) = \gr_t(z)$ for some cycle $z\in \CKh(D)$ whose homology class $[z]$ is nontorsion.

\begin{proposition}
\label{prop:st is invariant}
If $D$ and $D'$ are related by a Reidemeister move, then $\gr_t(D) = \gr_t(D')$. It follows that $\gr_t(D)$ (and hence $s_t(D)$) is an invariant of $L$.
\end{proposition}

\begin{proof}
In Bar-Natan's cobordism category, the cobordisms in the chain homotopy equivalence relating the complex for $D$ to the complex for $D'$ have degree zero after the appropriate quantum grading shifts for $D$ and $D'$ are introduced. By Lemma \ref{lem:grt filtered degree}, the induced isomorphism $\Kh(D) \to \Kh(D')$ preserves the filtration, so $\gr_t(D) = \gr_t(D')$.
\end{proof}

The following standard argument shows $s_t$ is a concordance invariant, essentially because the degree of cobordisms are entirely controlled by their Euler characteristic. 

\begin{theorem}
\label{thm:concordance-invariant}
If two knots $K_0$ and $K_1$ are concordant, then $\sat(K_0) = \sat(K_1)$.
\end{theorem}
\begin{proof}
Fixing a concordance $S \: K_0 \to K_1$ and diagrams $D_i$ for $K_i$, we obtain an $R$-linear map $S_* \: \Kh(D_0) \to \Kh(D_1)$. Pick $\zeta_0 \in \Kh(D_0)$, a nontorsion class with maximal $\gr_t$-grading, i.e.\ $\gr_t(\zeta_0) = \gr_t(D_0)$. By Corollary \ref{cor:induced map is injective mod torsion}, $S_*(\zeta_0)$ is nontorsion (in particular, nonzero) in $\Kh(D_1)$. Since $\chi(S) = 0$, Lemma  \ref{lem:grt filtered degree} and Proposition \ref{prop:st is invariant} imply that $\gr_t(D_0) = \gr_t(\zeta_0) \leq \gr_t(S_*(\zeta_0)) \leq \gr_t(D_1)$. Repeating the same argument with the reversed cobordism $\overline{S}: K_1 \to K_0$, we obtain $\gr_t(D_0) = \gr_t(D_1)$ (and hence $s_t(K_0) = s_t(K_1)$). 
\end{proof}

\begin{remark}
We can also define a related $t$-modified theory more akin to Ozsv{\'a}th-Stipsicz-Szab{\'o}'s  \cite{OSS-upsilon} original $\Upsilon$ formulation as follows. Let $\mathcal{P}$ denote the ring of long power series, consisting of formal sums $\sum\limits_{a\in A} r_a w^a$ where  $A\subset \R_{\geq 0}$ is well-ordered and $r_a\in \bF$. For $t\in [0,2]$, consider the ring homomorphism $\phi_t \: R \to \mathcal{P}$ given by $U\mapsto w^t, V\mapsto -w^{2-t}$. Note that $\phi_t(c) =w^t + w^{2-t} \neq 0$ for all $t\in [0,2]$.

For a link diagram $D$, form the complex $\CKh(D)\o_R \mathcal{P}$ obtained by base change via $\phi_t$. Alternatively, as in \cite{OSS-upsilon}, one can work over the ground ring $\bF[w^{1/n}]$ when $t=m/n$ is rational, with the understanding that $\bF[w^{1/n}]$ is naturally a subring of $\mathcal{P}$. A benefit of this approach is that the ground ring is a PID (or, when using $\mathcal{P}$, that it functions similarly to a PID in that finitely generated $\mathcal{P}$-modules decompose as direct sums of cyclic modules \cite[Section 11]{Brandal}). On the other hand, at $t=1$ we have $\phi_1(U+V) = 0$, causing a dramatic change in the differential. This could potentially lead to a discontinuity in an otherwise continuous invariant. 
\end{remark}

\subsection{Properties of the invariant}
\label{sec:properties}

We now list some properties of $s_t$ and discuss some examples and observations. 

\subsubsection{Immediate properties}

From the definition of $s_t$, some properties are immediate. For instance, the symmetry of the chain complex under swapping $U$ and $V$ yields the following symmetry in the function $s_t(L)$.

\begin{proposition}
\label{prop:symmetry}
For any link $L$ and $t\in [0,2]$, $s_t(L) = s_{2-t}(L)$. 
\end{proposition}

\begin{proof}
For any $\tau \in [0,2]$, let $\Kh_\tau(L)$ denote $\Kh(L)$ equipped with the filtration $F^\tau$. 
Consider the ring isomorphism $\sigma \: R\to R$ which interchanges $U$ and $V$. Fix a diagram $D$ of $L$ and extend $\sigma$ to a $\bF$-linear automorphism, also denoted $\sigma$, of $\CKh(D)$ by setting $\sigma(U^m V^n g) = U^n V^m g$ for any generator in $\Kg(D)$. Observe that $\gr_t(z) = \gr_{2-t}(\sigma(z))$ for any $z\in \CKh(D)$. Moreover, $\sigma \partial = \partial \sigma$, so $\sigma$ descends to a filtration-preserving $\bF$-linear isomorphism $\Kh_t(L) \xrightarrow{\sim} \Kh_{2-t}(L)$. Moreover, $\sigma$ sends $T(D)$ to $T(D)$, since if $\eta \in T(D)$ with $r\eta = 0$ for some $r\in R\setminus \{0\}$, then $\sigma(r) \sigma(\eta) = 0$ as well.   
\end{proof}

The following proposition explains the relationship between $s_t$ and Rasmussen's invariant \cite{Rasmussen, Turner, BN2}. 

\begin{proposition}
\label{prop:compare-s}
For a knot $K$, $s_0(K) = s_2(K) \leq s_{\bF}(L)$, where $s_{\bF}(K)$ is the Rasmussen $s$-invariant over $\bF$. 
\end{proposition}
\begin{proof}
Let $D$ be a diagram for $K$. 
When $t = 0$, we have $\gr_0 = \gr_q - 2 \gr_V$. 
Consider the filtration spectral sequence associated to the $\gr_U$ filtration grading. The associated graded complex has homology $E_1 \cong H_{BN}(D)$, where $H_{BN}(D)$ is the Bar-Natan homology of $D$, defined via the Frobenius algebra $\bF[V,X]/(X^2 = VX)$. As multiplication by $U$ does not affect $\gr_0$ grading, $\gr_0(D)$ is equal to the maximal $\gr_0$-grading within nontorsion homology classes $\eta$ with $\gr_U(\eta) = 0$. This is precisely $s_\bF(K)+1$. 

Since the total complex is bounded in homological grading, the spectral sequence collapses eventually.  As the filtration spectral sequence is a sequence of subquotients, the maximum $\gr_0$-grading of a nontorsion class on each page is non-increasing, and the result follows. 
The case for $t=2$ is similar, or follows from Proposition \ref{prop:symmetry}. 
\end{proof}

\begin{remark}
Let $K$ be a knot. We currently do not know if $s_0(-K,-o) = -s_0(K,o)$, where $-o$ is the orientation induced on  the mirror $-K$ by the orientation $o$ on $K$; see the discussion on mirrors in Section \ref{sec:mirrors}. 
If $s_t$ behaves as a concordance homomorphism, then we would know that $s_\bF = s_0 (= s_2)$, since Proposition \ref{prop:compare-s} would tell us that $s_0(K,o) = - s_0(-K,-o)\geq -s_\bF(-K,-o) = s_\bF(K,o)$.
\end{remark}

\begin{proposition}
If $t \in \Q$ with $t = \frac{p}{q}$, then $s_t(L) \in \frac{1}{q}\Z$ for any link $L$. 
\begin{proof}
This follows from the fact that the distinguished generators $U^m V^n y\in \Kg(D)$ only have $\gr_t$-gradings valued in $\frac{1}{q}\Z$.
\end{proof}
\end{proposition}

\begin{proposition}
For knots $K_1$ and $K_2$, $\sat(K_1 \coprod K_2) = \sat(K_1) + \sat(K_2) +1$.
\begin{proof}
Fix $t$.
Pick cycles $\eta_i \in \CKh(K_i)$ such that $\gr_t(\eta_i) = \sat(K_i)+1$, i.e.\ $\eta_i$ maximizes $\gr_t$ inside its homology class $[\eta_i]$, and $[\eta_i]$ maximizes $\gr_t$ across all homology classes. 
Since the differential of $\CKh(K_1 \coprod K_2)$ does not interact with the tensor product, $\gr_t(\eta_1 \otimes \eta_2)$ also maximizes $\gr_t$ in its homology class. Thus $\gr_t(\eta_1 \otimes \eta_2) = \gr_t(\eta_1) + \gr_t(\eta_2)$, which after the normalization gives the desired formula in terms of $\sat$. 
\end{proof}
\end{proposition}

\begin{proposition} For knots $K_1$ and $K_2$, 
$| \sat(K_1 \# K_2) - (\sat(K_1) + \sat(K_2) + 1) | \leq 1$. 
\begin{proof}
Let $C \: K_1 \coprod K_2 \to K_1 \# K_2$ be the saddle cobordism that realizes the connected sum operation. Since $C$ is connected, Corollary \ref{cor:induced map is injective mod torsion} implies that the induced map on homology $C_* \: \Kh(K_1 \coprod K_2) \to \Kh(K_1 \# K_2)$ is injective modulo torsion. Since $\chi(C) = -1$, we have $\sat(K_1 \# K_2) \geq \sat(K_1 \coprod K_2)-1$. Considering the reverse cobordism $\b{C} \: K_1 \# K_2 \to K_1 \coprod K_2$, we find that $\sat(K_1 \coprod K_2) \geq \sat(K_1 \# K_2)-1$. Combining these, we have the desired inequality.
\end{proof}
\end{proposition}

\subsubsection{Continuity in $t$}

Our next task is to show that $s_t(L)$ is continuous as a function of $t$. In fact, we show the following:
\begin{theorem}
\label{thm:continuity}
For a link $L$, $s_t(L) \: [0,2] \to \R$ is piecewise-linear continuous in $t$, and is differentiable at all but finitely many points.
\end{theorem}

We make the following nonstandard definitions in order to streamline the proof. First, observe that the tri-grading $( \gr_U, \gr_V,\gr_q)$ of a homogeneous generator $U^mV^n g \in \Kg(D)$ lies in the lattice $\Z^3$ in  $( \gr_U, \gr_V,\gr_q)$-space.
\begin{itemize}

    \item A lattice point $\bp = (m,n,q) \in \Z^3$ is \emph{occupied by a chain $x \in \CKh(D)$} if $x$ has a homogeneous summand (with nonzero coefficient) $U^mV^ng$ with $\gr_q(g) = q$. Note that a chain occupies finitely many lattice points.  Sometimes we will abuse notation and write $\gr_t(m,n,q)$ to mean the $\gr_t$-grading of a homogeneous generator $U^mV^ng$ that has trigrading $(m,n,q)$. The function $\gr_t(m,n,q) = q - t\cdot m - (2-t)\cdot n$ is linear in $t$.
    
    \item A subset $S$ of the lattice $\Z^3$ \emph{supports} a chain $x$ if $x$ only occupies lattice points in $S$. 
    
    \item A lattice point $\bp \in \Z^3$ is \emph{occupied by $\CKh(D)$} if there is some chain $x \in \CKh(D)$ that occupies $\bp$. Note that the set points occupied by $\CKh(D)$ is bounded in $\gr_q$, and bounded from below in $\gr_U$ and $\gr_V$ (see Figure \ref{fig:filtration is discrete}). Given a chain $x$, $\gr_t(x)$ is the minimal $\gr_t$-grading among all homogeneous summands of $x$. Therefore, as a function of $t$,  $\gr_t(x)$  is piecewise-linear and continuous on $[0,2]$. 
    
    \item Given $t \in [0,2]$, a lattice point $\bp$ is \emph{measured by $\gr_t(D)$} if there is a chain $x$ representing a nontorsion homology class such that $\gr_t(x) = \gr_t(D)$ (i.e.\ the homology class of $x$ maximizes $\gr_t$ among nontorsion classes) such that $x$ occupies $\bp$. Note that this means $x$ has a homogeneous summand at $\bp$, and that the $\gr_t$ of this summand is minimal among the homogeneous summands of $x$. Also, $\gr_t(D)$ may measure multiple lattice points at once; if $t \in (0,2)$, $\gr_t(D)$ may only measure finitely many lattice points, due to the boundedness of the set of lattice points occupied by $\CKh(D)$ and the position of the level sets (planes) of the $\gr_t$ function.
    
    \item Given $t \in [0,2]$, we are most interested in the maximal filtration level containing representatives of nontorsion homology classes, so we denote this level $F^t_\star$. (In other words, $F^t_\star = F^t_{\gr_t(D)}$.) The \emph{boundary of $F^t_\star$}, denoted $\partial F^t_\star$, is the plane in $(\gr_U, \gr_V,\gr_q)$-space given by the equation 
    \[
        \gr_q - t \cdot \gr_U - (2-t) \cdot \gr_V = \gr_t(D).
    \]
    Chains in $F^t_\star$ can only occupy lattice points $(m,n,q)$ where  $q - t \cdot m - (2-t) \cdot n \geq \gr_t(D)$.
    
\end{itemize}

To help visualize the calculation of $\gr_t(D)$, we make the following observations.

\begin{itemize}

    \item Suppose that at time $t \in [0,2]$, the lattice point $\bp = (m, n,q)$ is being measured. The plane $\{\gr_t = \gr_t(\bp)\}$ is given by the dot product equation 
    \[
        (-t, -(2-t), 1) \cdot ( \gr_U - m, \gr_V - n, \gr_q - q ) =0,
    \]
    or equivalently,
    \[
        t(\gr_U - m) + (2-t)(\gr_V - n) = \gr_q - q.
    \]
    For $q'\in \Z$, the intersection of the above plane $\{\gr_t = \gr_t(\bp)\}$ and the plane $\{\gr_q = q'\}$ is the line 
    \[
       l(\bp,t,q') = \{ t(\gr_U - m) + (2-t)(\gr_V - n) = q' - q\}
    \]
    of slope $\frac{t}{t-2} \in [-\infty, 0]$
    containing the \emph{pivot point}
    \[ 
        \left ( m + \frac{q'-q}{2}, n + \frac{q'-q}{2}, q'\right ).
    \]
    Note that the pivot point is independent of $t$. 
    The collection of pivot points at each $\gr_q$ level is given by the line 
    \[
        \left ( m + \frac{\gr_q-q}{2}, n + \frac{\gr_q-q}{2}, \gr_q \right ).
    \]
    \item As $t$ increases from $0$ to $2$, the line $l(\bp,t,q')$ in the plane $\{\gr_q = q'\}$ starts out with slope $0$ and rotates clockwise through negative slopes, about the pivot point $( m + \frac{q'-q}{2}, n + \frac{q'-q}{2}, q')$. At $t = 2$, the line $l(\bp, t,q')$ is vertical. 
    
    \item The function $\gr_t(D)$ is always measuring the $\gr_t$ of some homogeneous generator $U^mV^ng \in \Kg(D)$. That is, if $\eta$ is a nontorsion homology class in $\Kh(D)$ with maximal $\gr_t$-grading, then there is some chain representative $x = \sum c_j U^{m_j}V^{n_j}g_j$ achieving $\gr_t(x) = \gr_t(\eta)$, and this value is the minimum $\gr_t$ over all the homogeneous summands $U^{m_j}V^{n_j}g_j$ of $x$. As such, $\gr_t(D)$ is piecewise-linear: each time $t$ in the domain is either a critical point (a point of nondifferentiability) or in the interior of an interval where the $\gr_t(D)$ function is linear (a time interval on which the lattice point being measured is constant). 
    
\end{itemize}

\begin{proof}[Proof of Theorem \ref{thm:continuity}]

We will show that $\gr_t(D)$ is continuous for a diagram $D$ of a link $L$. We mainly make use of the discreteness of the integer lattice in $ (\gr_U, \gr_V,\gr_q)$-space.

Suppose $\tau \in (0,2)$. 
The plane $\partial F^\tau_\star$ intersects at least one lattice point (as it is measuring some class) and at most finitely many lattice points (by boundedness of the occupied region). Label the points in this finite set $\{\bp_i\}_{0\leq i \leq k}$.
By the discreteness of the occupied lattice, for each $\bp_i$, there exists a $\delta_i >0$ such that for $t \in (\tau -\delta_i,\tau) \cup (\tau, \tau + \delta_i)$, the plane
$\{\gr_t = \gr_t(\bp_i)\}$
intersects the occupied lattice only at $\bp_i$.
Let $\delta$ be the minimum of this finite set $\{\delta_i\}$. Then $\gr_t(D)$ is piecewise-linear and continuous on $(\tau-\delta, \tau+\delta)$, possibly with a corner at $t = \tau$. 
This shows $\gr_t(D)$ can have only finitely many corners on any compact interval $[a,b] \subset (0,1)$. 

Next, we will show that $\gr_t$ is right-continuous at $t=0$; left-continuity at $t=2$ follows from symmetry (Proposition \ref{prop:symmetry}).

Choose a cycle $x$ representing a nontorsion class $[x]$ such that $\gr_0(x) = \gr_0([x]) = \gr_0(D)$. The cycle $x$ occupies finitely many lattice points $\{\mathbf{x}_i\}$. Let us first consider the boundaries of the filtration levels $F^t_{\gr_t(x)}$  measuring $x$ at times $t$; these are by definition always tangent to the convex hull of the points $\{\mathbf{x}_i\}$. 
There is some $\delta^l$ (for \emph{left}; see Figure \ref{fig:rotate-plane}) such that for $t \in (0,\delta^l)$, the filtration level $F^t_{\gr_t(x)}$   does not pick up any new lattice points; this because of the boundedness of the ``left" side of the lattice (i.e.\ the side with bounded $\gr_V$ gradings). 
There is also some $\delta^r$ (for \emph{right}, see Figure \ref{fig:rotate-plane}) and a lattice point $\bf{x}_j$ occupied by $x$ such that for all $t \in [0,\delta^r)$, $\partial F^t_{\gr_t(x)}$ is in fact always tangent to $\mathbf{x}_j$ (there may be multiple $\mathbf{x}_j$; pick any one). See Figure \ref{fig:delta r} for a schematic depiction of the choice of $\mathbf{x}_j$ and $\delta^r$.

Let $\delta = \min\{\delta^l, \delta^r\}$. If $\gr_t(D) = \gr_t(x)$ for all $t \in [0,\delta)$, then we are done, as $\gr_t(x)$ is piecewise-linear continuous with finitely many critical points. 
So, suppose that there is some cycle $y \neq x$ representing some nontorsion homology class (possibly $[x]$, or possibly some other nontorsion class) such that at some time $\tau \in (0,\delta)$, $\gr_\tau(y)$ is \emph{strictly greater than} $\gr_\tau(x)$. Let $\{\mathbf{y}_i\}$ be the finite set of lattice points occupied by $y$. 

Since $F^t_{\gr_t(x)}$ does not pick up any new lattice points on the interval $t \in [0,\tau]$,
it must be that all points $\{ \mathbf{y}_i\}$ lie in the region occupied by $F^t_{\gr_t(x)}$ for all $t \in [0,\tau]$. 
Since $\gr_0(x) = \gr_0(D)$, we must have $\gr_0(y) = \gr_0(x)$; otherwise $\gr_0(y) = \gr_0(D) > \gr_0(x)$. 
Furthermore, it must be that the lattice points $\mathbf{y}_j$ being measured by $\gr_0$ lie in the finite \emph{left} side of the line $L(\mathbf{x}_j)$ in the plane $\partial F^0_\star = \partial F^0_{\gr_0(x)}$, and in particular cannot lie on $L(\mathbf{x}_j)$ (because $\gr_\tau(y) >\gr_\tau(x)$). 

We now repeat the argument, replacing $x$ with $y$, and possibly obtaining an ``even better" cycle $y'$ representing a nontorsion homology class, and so on. This process must terminate after at most $\gr_V(\mathbf{x}_j)+1$ iterations, and we have that $\gr_t(D)$ is equal to $\gr_t(z)$ for some nontorsion cycle $z$ on an interval $[0,\delta_0)$. 

\begin{figure}[htp]
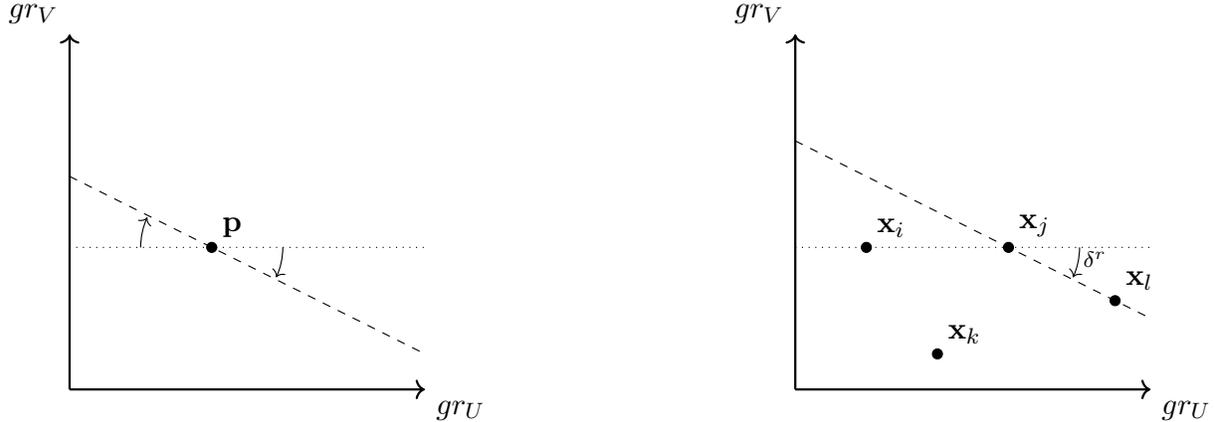

\subcaptionbox{The boundary planes of the filtration levels measuring lattice point $p$, at two different times $t = 0$ (dotted) and $t = \tau >0$ (dashed), intersect the plane $\gr_q = \gr_q(p)$ as shown. The dotted line indicates the boundary of the filtration level at $t = 0$. As $t$ increases, the boundary rotates clockwise. The wedge northwest of $\bp$ that is bounded by the dotted and dashed lines contains finitely many occupied lattice points, whereas the wedge southeast of $\bp$ is infinite. \label{fig:rotate-plane}}
{ \includegraphics[width=2.6in]{images/rotate-plane.tex}
}\hfill%
\subcaptionbox{ A two-dimensional depiction of choosing $\mathbf{x}_j$ and $\delta^r$. The distinguished cycle $x$ occupies finitely many lattice points. We can choose one, $\mathbf{x}_j$, such that, for some $\delta^r >0$ (depending on $\mathbf{x}_j$), the boundary of the filtration level $\partial F^t_{\gr_t(x)}$ is tangent to $\mathbf{x}_j$ for all $t\in [0,\delta^r)$. In particular, $x\in F^t_{\gr_t(x)}$ for all $t\in [0,\delta^r)$. \label{fig:delta r}}
{\includegraphics[width=2.6in]{images/delta_r.tex}
}%
\caption{Figures illustrating parts of the proof of Theorem \ref{thm:continuity}.}
\end{figure}

Since $\gr_t(z)$ has finitely many critical points, we now know that $\gr_t(D)$ has finitely many critical points on $[0,\delta_0)$ and $(2-\delta_0, 2]$. As $\gr_t(D)$ also has finitely many critical points on the compact interval $[\delta_0/2, 2-\delta_0/2]$, the statement follows.
\end{proof}

\subsection{Examples}

In this section, we discuss a few preliminary examples to help elucidate the potential behavior of $s_t$. At present, we have not implemented a program to compute $s_t$. Using such a program, one could investigate which types of knots have ``trivial" $s_t$, in the sense that $\frac{ds_t}{dt}$ is constant on the interval $[0,1]$. We are also interested in finding an example where $s_t$ has a change in slope at some $t^* \in (0,1)$, and exploring the topological meaning of the value of $t^*$.

\begin{example}
If $\mathcal{U}$ is the unknot, then $s_t(\mathcal{U}) = 0$. 
\end{example}

\begin{example}
Let $(K,o)$ be a positive oriented knot; let $D$ be a positive diagram of $K$ with $n$ crossings.
The complex $\CKh(D)$ is supported on homological gradings $0$ through $n$.
The oriented resolution $D_o$ is the unique resolution in the lowest supported homological grading $\gr_h = 0$.
Let $r$ denote the number of Seifert circles in the oriented resolution $D_o$. Applying Algorithm \ref{alg:ABcolors}, let $a$ denote the number of $\colora$ circles and $b$ the number of $\colorb$ circles in $D_o$; hence $r = a+b$. 

Recall that in $\CKhD(D)$, we have two canonical Lee generators $\s_o$ and $\s_{\obar}$ (see the discussion following Algorithm \ref{alg:ABcolors}). We may define lifts $\tilde \s_o$ and $\tilde \s_{\obar}$ in $\CKh(D)$ as follows.
    For $\tilde \s_o$, $\colora$-labeled circles contribute the tensor factor  $X-U$ and $\colorb$-labeled circles contribute the tensor factor $V-X$; for $\tilde \s_{\obar}$ these assignments are reversed: $\colora$ circles contribute $V-X$ and $\colorb$ circles contribute $X-U$. 
In other words, if viewed as chains in the localized complex $\CKhD(D)$,
$\tilde \s_o = (V-U)^r \s_o$ and $\tilde \s_{\obar} = (V-U)^r \s_{\obar}$.

As $D_o$ is at the lowest supported homological grading of $\CKh(D)$, there are no boundaries, so each cycle at $D_o$ is unique in its homological class. Such classes must be nontorsion. To find the class with the maximal $\gr_t$-grading, we use the fact that the localized homology $\KhD(D)$ is generated by $[\s_o]$ and $[\s_{\obar}]$. Below, we will define a cycle $\gamma$ at $\gr_h=0$ in $\CKh(D)$ whose image in $\CKhD(D)$ represents a class that lies in the $R$-span of $\{[\s_o], [\s_{\obar}]\}$. We will also see that $(V-U)$ divides $\gamma$ exactly once, and that $\gamma/(V-U)$ maximizes $\gr_t$-grading and therefore determines $s_t(K)$. 

Fix an ordering of the circles in $D_o$ so that all the $\colora$ circles appear first in the tensor product, followed by all the $\colorb$ circles. Then $\tilde \s_o = (X-U)^{\otimes a} \otimes (V-X)^{\otimes b}$ and $\tilde \s_{\obar} = I ( \tilde \s_o) = (V-X)^{\otimes a} \otimes (X-U)^{\otimes b}$.

As a function of $t$, the $\gr_t$-grading of $\tilde \s_o$ and $\tilde \s_{\obar}$ is the constant function $n-r$, given by measuring the lowest grading homogeneous summand $\pm X^{\otimes r}$. But this is not the highest $\gr_t$-grading achieved by a nontorsion class; consider the linear combination $\gamma$ given by
\[
    \gamma = 
    \begin{cases}
    \tilde \s_o - \tilde \s_{\obar} & \text{if }  a-b \equiv 0 \mod 2 \\
    \tilde \s_o + \tilde \s_{\obar} & \text{if } a-b \equiv 1 \mod 2
    \end{cases}.
\]
In each case, the term $\pm X^{\otimes r}$ in the generators cancel, leaving homogeneous summands with strictly higher quantum gradings. 

The terms in $\gamma$ can be easily described, as they come in pairs of homogeneous summands from $\tilde \s_o$ and $\tilde \s_{\obar}$.
Recall that under the chosen ordering of components of $D_o$, $\tilde \s_o$ may be written $(X-U)^{\otimes a} \otimes (V-X)^{\otimes b}$. Expanding this product as the sum of monomials (pure tensors in $1$ and $X$ with a monomial in $U$ and $V$ for the coefficient), we see that $\tilde \s_o$ is the sum of terms which we will describe below.

For $0 \leq k \leq r$, consider the monomials where $k$ circles are labeled by $1$ rather than $X$. For $0 \leq j \leq \min\{a,b\}$, we label $j$ of the $\colora$ circles with $-U$ and $k-j$ of the $\colorb$ circles $V$. The rest of the $\colora$ circles are labeled $X$ and the rest of the $\colorb$ circles are labeled $-X$.
Up to a particular permutation $\rho$ of tensor factors, we may write this summand as 
\begin{align*}
    & X^{\otimes a-j} \otimes (-U)^{\otimes j} 
    \otimes 
    (-X)^{\otimes b -(k-j)}
    \otimes V^{\otimes k-j} \\
    & = (-1)^{\otimes b-k}\cdot U^j V^{k-j}\cdot X^{\otimes a-j} \otimes 1^{\otimes j} \otimes X^{\otimes b - (k-j)} \otimes 1^{\otimes k-j}.
\end{align*}
Applying $I$ to the above summand, we obtain the corresponding summand in $\tilde \s_{\obar}$, which under the same permutation $\rho$ of tensor factors can be written 
\begin{align*}
    & (-X)^{\otimes a - j} 
    \otimes V^{\otimes j}
    \otimes X^{\otimes b-(k-j)}
    \otimes (-U)^{\otimes k-j} \\
    & = (-1)^{a+k}\cdot V^j U^{k-j} \cdot X^{\otimes a-j} \otimes 1^{\otimes j} \otimes X^{\otimes b - (k-j)} \otimes 1^{\otimes k-j}.
\end{align*}
(Note that the quantum grading of both of these summands is $-n+r+2k$.)

In both cases of the parity of $a-b$, we have that $\gamma$ is the sum of pure tensors in $X$ and $1$ with coefficients of the form $\pm( U^j V^{k-j} - V^j U^{k-j})$.
Furthermore, we see that $(V-U)$ divides $\gamma$ exactly once, for the $k=1$ summand pairs are only divisible once.

We claim that $\gr_t(\gamma/(V-U)) = \gr_t(D)$. To see this, observe that the only $R$-linear combinations of $\tilde \s_o$ and $\tilde \s_{\obar}$ (in $\CKh(D)$) that do not have an $X^{\otimes r}$ summand are multiples of $\gamma$; those that are not multiples of $\gamma$ have $\gr_t$ equal to $\gr_t(X^{\otimes r})$, which is less than $\gr_t(\gamma)$. Furthermore, $(V-U)$ divides $\gamma$ once, so $\gamma/(V-U)$ is indeed a cycle in $\CKh(D)$, and for all $t$, $\gr_t(\gamma/(V-U)) \geq \gr_t(\gamma)$. Finally, since $\gamma/(V-U)$ is unique in its homology class, we have $\gr_t(D) = \gr_t([\gamma]) = \gr_t(\gamma)$. This is equal to the quantum grading of the pure tensor $X^{\otimes r-1} \otimes 1$, for example, which is $n-r+2$. Hence $s_t(K) = n-r+1$, which is equal to the Rasmussen invariant $s_\bF(K)$.

\end{example}

While we have yet to implement a program to compute this invariant and find a concrete example where $\sat$ is nonconstant, it is possible to describe when $\sat$ could be nontrivial, and explain what a change in slopes represents, by looking at the following example. 

\begin{example}
In the complex for the standard braid-closure diagram of $T(3,-4) = \widehat{(\sigma_1\inv\sigma_2\inv)^4}$, consider the differential from the resolutions at $\gr_h = -1$ to the unique oriented resolution at $\gr_h = 0$.
The eight resolutions at $\gr_h = -1$ fall into two types: (a) those with an unoriented resolution of a $\sigma_1\inv$ crossing,  and (b) those with an unoriented resolution of a $\sigma_2\inv$ crossing. The differentials out of these resolutions are easily computed, and in particular include the following components (up to an overall sign):

\begin{enumerate}
    \item[(a)] Circle splitting due to a change of resolution of a $\sigma_1\inv$ crossing:
    \begin{equation*}
        \begin{aligned}
          \begin{tikzpicture}[scale=.3]

\begin{turn}{90}
\coordinate (e1) at (-0.707,-0.707);
\coordinate (e2) at (0.707,-0.707);
\coordinate (e3) at (-1.414,-1.414);
\coordinate (e4) at (1.414,-1.414);

\draw[] (0,0) circle [radius=3];

\draw[] (e2) arc[start angle=-45, end angle=225, radius=1];
\draw[] (e2) arc[start angle=135, end angle=315, radius=.5];

\draw[] (e4) arc[start angle=-45, end angle=225, radius=2];
\draw[] (e3) arc[start angle=-135, end angle=45, radius=.5];
\end{turn}

\coordinate (p2) at (12,0);

\draw[] (p2) circle [radius=1];
\draw[] (p2) circle [radius=2];
\draw[] (p2) circle [radius=3];

\draw[->] (3.5,0) -- (8.5,0);
\end{tikzpicture}
        \end{aligned}
    \end{equation*}

    \begin{equation}
    \label{eqn:U+Vexample}
        1 \otimes X \mapsto 
        1 \otimes X \otimes X  + X \o 1 \o X -(U+V)1 \otimes 1 \otimes X,
    \end{equation}
    \begin{equation}
    \label{eqn:UVexample}
        X \otimes 1 \mapsto 
        X \otimes X \otimes 1 -UV 1 \otimes 1 \otimes 1.
    \end{equation}
    In the above maps, tensor factors are ordered according to the innermost-to-outermost ordering on the depicted circles.
    \item[(b)] Circle splitting due to a change of resolution of a $\sigma_2\inv$ crossing:
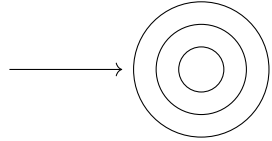
        \begin{equation*}
        \begin{aligned}
          \begin{tikzpicture}[scale=.3]

\begin{scope}[shift={(-24,0)}]
\begin{turn}{-90}
\coordinate (p4) at (24,0);

\coordinate(l2') at (-1.5 + 24, 2.598 );
\coordinate (l1') at (-1+24, 1.732);
\coordinate(r2') at (1.5 + 24, 2.598 );

\draw[] (p4) circle [radius=1];

\draw[] (r2') arc[start angle = 60, end angle = 240, radius =.5];
\draw[] (l1') arc[start angle = -60, end angle=120, radius =.5];

\draw[] (l1') arc[start angle = -240, end angle =60, radius = 2];
\draw[] (l2') arc[start angle = -240, end angle =60, radius = 3];
\end{turn}
\end{scope}

\coordinate (p2) at (12,0);

\draw[] (p2) circle [radius=1];
\draw[] (p2) circle [radius=2];
\draw[] (p2) circle [radius=3];

\draw[->] (3.5,0) -- (8.5,0);
\end{tikzpicture}
        \end{aligned}
    \end{equation*}
    
    \begin{equation*}
        1 \otimes X \mapsto
        1 \otimes X \otimes X - UV 1 \otimes 1 \otimes 1
    \end{equation*}
    \begin{equation*}
        X \otimes 1 \mapsto 
        X \otimes X \otimes 1 + X \otimes 1 \otimes X - (U+V) X \otimes 1 \otimes 1
    \end{equation*}      
\end{enumerate} 
All chains at $\gr_h = 0$ are cycles: since all crossings are negative, $\gr_h = 0$ is the maximum homological grading in the complex. 

From Equation \ref{eqn:U+Vexample}, we see that the cycle $1 \otimes X \otimes X + X \otimes 1 \otimes X$ at $\gr_t$-grading $-9$ is homologous to the cycle $(U+V)\cdot 1 \otimes 1 \otimes X$ at $\gr_t$-grading $\min\{-7-t, -9+t\}$. 

From Equation \ref{eqn:UVexample}, we see that the cycle $X \otimes X \otimes 1$ at $\gr_t$-grading $-9$ is homologous to the cycle $UV \cdot (1 \otimes 1 \otimes 1)$ at $\gr_t$-grading $-7$.
In this example, one can compute that the cycle $1 \otimes 1 \otimes 1$ at $\gr_h=0$ and $\gr_t = -5$ is nontorsion (by showing that it is homologous to $(\tilde \s_o + \tilde \s_{\obar})/(V-U)$), so $s_t(T(3,-4)) = -6 = s(T(3, -4))$; nevertheless, the above cycles give an example of potential behavior of the invariant for a knot that is not positive or negative. 
\end{example}

\section{A reduced theory}
\label{sec:reduced}

In this section we define a reduced version of  $U(1) \times U(1)$ equivariant Khovanov homology. Self-duality of the Frobenius pair $(R,A)$ and mirror images are also discussed.
We envision the reduced theory as particularly useful in the study of knots, as opposed to links in general. Throughout this section, we use $L$ to denote a link, whereas $K$ denotes a knot.

\subsection{The basepoint action}
Let $L$ be a \emph{based} link, i.e.\ marked with a basepoint $p \in L$. Taking a generic projection we obtain a link diagram $D$ of $L$ with a basepoint $p$ away from the crossings. Consider the ordinary ($U=V=0$) Khovanov complex $\CKh_0(D)$, built via the Frobenius algebra $A_0=\Z[X]/(X^2)$. There is a \emph{basepoint action} $\xi_p$ of $A_0$ on $\CKh_0(D)$; see \cite[Section 3]{Kh-patterns}. The action is given my merging a small circle labeled $X$ near $p$ with the marked component at each resolution $D_u$. The \emph{reduced Khovanov complex} $\CKhred_0(D
,p) $ is defined to be 
\[
\CKhred_0(D,p) = \im(\xi_p)\{1\},
\]
the ($\gr_q$-shifted) image of the action of $X$ via $\xi_p$. 
Equivalently, in a resolution $D_u$, let $Z_{u,p} \subset D_u$ denote the circle containing $p$. Then $\CKhred_0(D,p)$ is freely generated by the Khovanov generators in each $D_u$ which mark $Z_{u,p}$ by $X$, together with an overall upward $\gr_q$-shift of $1$.

\begin{remark}
Turner \cite{Turner} studied a reduced version of Bar-Natan homology, but to our knowledge there is not yet a construction for reduced $U(2)$-equivariant Khovanov homology. This is further evidence that expanding the ground ring to include the roots of $X^2 - E_1X + E_2$ may provide a richer algebraic environment for mining topological information. 
\end{remark}

\begin{remark}
A careful study of the basepoint action on $U(1) \times U(1)$ equivariant homology is present in \cite{Gujral-ribbon-distance}, where Gujral uses the  $U(1)\times U(1)$ theory (referred to as $\alpha$-homology therein) to give ribbon distance bounds, following work of Alishahi \cite{Alishahi}, Alishahi-Dowlin \cite{ADunknotting}, and Sarkar \cite{Sarkar}. 
\end{remark}

\subsection{The reduced equivariant chain complex} \label{sec:reduced-complex}

Let $(L,o)$ be an oriented link, $D$ a diagram for $(L,o)$, and $p$ a marked point on $D$.
The equivariant chain complex $\CKh(D)$ has an action by $A$ similar to the basepoint action above: define an endomorphism $f_p$ of $\CKh(D)$ to be multiplication by $X-U$ on the tensor factor of $\F(D_u)$ corresponding to the marked circle $Z_{u,p}$, and identity on all other tensor factors. In other words, the action $f_p$ merges a circle labeled $X-U$ into the marked circle. 

\begin{definition}
The \emph{reduced complex} $\CKhred(D,p)$ of a based knot diagram $(D,p)$ is defined to be $\CKhred(D,p) = \im(f_p) \{1\}$. Its homology is denoted $\Khred(D,p)$.
\end{definition}

Both $\{1, X-U\}$ and $\{1, X-V\}$ form a basis for $A$. We have
\begin{align}
\label{eq:mult and comult on X-U, X-V}
    \begin{aligned}
   (X-U) (X-U) &= (V-U)(X-U) & \qquad &  \Delta(X-U) = (X-U) \o (X-U)\\
    (X-V) (X-V) &= (U-V) (X-V) & \qquad & \Delta(X-V) = (X-V) \o (X-V).
    \end{aligned}
\end{align}

It follows that the $R$-submodule of $\CKh(D,p)$ that is freely generated by the Khovanov generators for which the marked circle $Z_{u,p}$ in each resolution is labeled by $X-U$ forms a subcomplex of $\CKh(D)$, which (after a $\grq$-grading shift) is equal to $\CKhred(D,p)$. 

Note that symmetry between $U$ and $V$ is broken by declaring the action of $f_p$ to be multiplication by $X-U$. Let $f_p'$ be the endomorphism of $\CKh(D)$ given by multiplication by $X-V$ near the basepoint, and let 
$\CKhred'(D,p) = \im(f_p')\{1\}$. Recall that the algebra involution $\iota$ from Equation \eqref{eq:iota involution} interchanges $X-U$ and $V-X$, so the corresponding involution $I$ of $\CKh(D)$ induces an isomorphism $\CKhred(D,p) \rar{\sim} \CKhred'(D,p)$. 
\begin{proposition}
\label{prop:reduced invariance}
Let $(L, p)$ be a based knot with a based diagram $(D,p)$. The chain homotopy type of $\Khred(D,p)$ is an invariant of $(L,p)$. 
\begin{proof}
The argument is given in \cite[Section 3]{Kh-patterns} (see also \cite[Lemma 2.2]{BLS}), but we repeat it here. Any two based diagrams of $(L,p)$ are related by a sequence of Reidemeister moves away from the marked points and a move which slides the basepoint past a crossing. This second basepoint sliding move can be realized as a sequence of Reidemeister moves away from basepoints, followed by the global move shown in Figure \ref{fig:pt at infinity}, and finally followed by another sequence of Reidemeister moves away from basepoints. See Figure \ref{fig:around-town} for a schematic.

It is straightforward to see that the chain homotopy equivalences for these Reidemeister moves preserve the respective reduced complexes. For the global move, use the obvious identification of circles in each resolution before and after the move; the associated isomorphism of chain complexes clearly preserves the respective reduced complexes.
\end{proof}
\end{proposition}

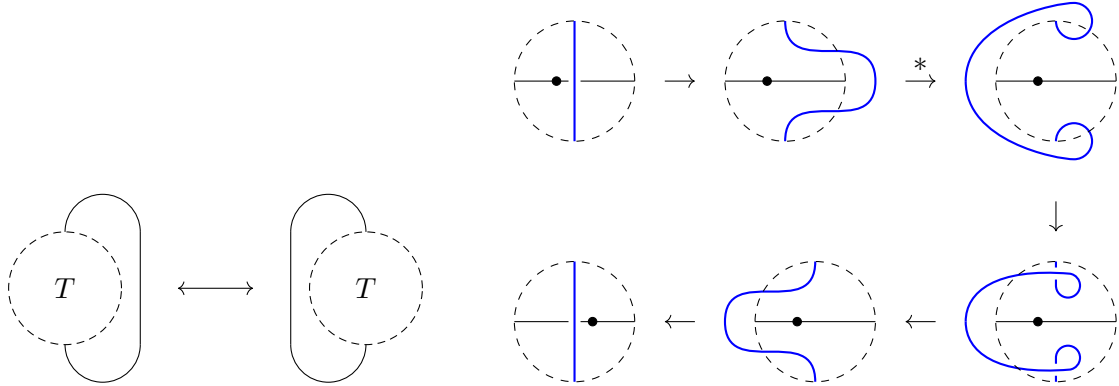
\begin{figure}
\subcaptionbox{The global move, where $T$ is an oriented, based $(1,1)$-tangle. \label{fig:pt at infinity}}[.35\linewidth]{  \begin{tikzpicture}
\draw[densely dashed] (0,0) circle [radius=.75];

\node at (0,0) {$T$}; 
\draw  (1,.75) arc (0:180:.5 and .5);
\draw (1,.75) -- (1,-.75); 
\draw  (0,-.75) arc (180:360:.5 and .5);

\begin{scope}[shift={(4,0)}]
\draw[densely dashed] (0,0) circle [radius=.75];

\node at (0,0) {$T$}; 
\draw  (0,.75) arc (0:180:.5 and .5);
\draw (-1,.75) -- (-1,-.75); 
\draw  (-1,-.75) arc (180:360:.5 and .5);
\end{scope}

\draw[<->] (1.5,0) -- (2.5,0);

\end{tikzpicture}}
\subcaptionbox{To move the basepoint past a crossing, when the basepoint is located on the under-strand, pull the the blue strand out of the page and above the rest of the diagram. The arrow marked $*$ is the global move. For the over-strand case, perform the same move, except with the blue strand pulled into the page, below the rest of the diagram.  \label{fig:around-town}}[.6\linewidth]{
    \begin{tikzpicture}[scale=.8]

\draw[dashed] (0,0) circle (1cm);
\draw (-1,0) -- (-.1,0);
\draw (.1,0) -- (1,0);
\draw[blue, thick] (0,1) -- (0,-1);
\filldraw (-0.3,0) circle (2pt);
\draw[->] (1.5,0) -- (2,0);

\begin{scope}[xshift=3.5cm]
\draw[blue, thick] (0,1) 
    .. controls (0,0) and (1.5,1) .. 
    (1.5,0)
    .. controls (1.5,-1) and (0,0) ..
    (0,-1);    
\draw[dashed] (0,0) circle (1cm);
\draw (-1,0) --(1,0);
\filldraw (-0.3,0) circle (2pt);
\draw[->] (2,0) -- (2.5,0) node[midway,above] {$*$};
\end{scope}

\begin{scope}[xshift=8cm]
\draw[thick, blue] (0,1) 
    arc (-180:90:.3cm)
    .. controls (0,1.3) and (-1.5,1) ..
    (-1.5,0);
\begin{scope}[yscale=-1]
\draw[thick, blue] (0,1) 
    arc (-180:90:.3cm)
    .. controls (0,1.3) and (-1.5,1) ..
    (-1.5,0);
\end{scope}
\draw[dashed] (0,0) circle (1cm);
\draw (-1,0) --(1,0);
\filldraw (-0.3,0) circle (2pt);
\draw[->] (0,-2) -- (0,-2.5) ;
\end{scope}


\begin{scope}[xshift=8cm, yshift=-4cm]
\draw[blue,thick] (0,1)-- (0,.9);
\draw[blue, thick] (0,.7) 
    --(0,.6) 
    arc (-180:80:.2cm)
    .. controls (-1, .9) and (-1.5, .5) ..
    (-1.5,0);
\begin{scope}[yscale=-1]
\draw[blue,thick] (0,1)-- (0,.9);
\draw[blue, thick] (0,.7) 
    --(0,.6) 
    arc (-180:80:.2cm)
    .. controls (-1, .9) and (-1.5, .5) ..
    (-1.5,0);
\end{scope}
\draw[dashed] (0,0) circle (1cm);
\draw (-1,0) --(1,0);
\filldraw (-0.3,0) circle (2pt);
\draw[->] (-2,0) -- (-2.5,0);
\end{scope}

\begin{scope}[xshift=4cm,xscale=-1, yshift=-4cm]
\draw[blue, thick] (0,1) 
    .. controls (0,0) and (1.5,1) .. 
    (1.5,0)
    .. controls (1.5,-1) and (0,0) ..
    (0,-1);    
\draw[dashed] (0,0) circle (1cm);
\draw (-1,0) --(1,0);
\filldraw (0.3,0) circle (2pt);
\end{scope}

\begin{scope}[yshift=-4cm]
\draw[dashed] (0,0) circle (1cm);
\draw (-1,0) -- (-.1,0);
\draw (.1,0) -- (1,0);
\draw[blue, thick] (0,1) -- (0,-1);
\filldraw (0.3,0) circle (2pt);
\draw[->] (2,0) -- (1.5,0);
\end{scope}

\end{tikzpicture}
   }
   \caption{A depiction of the chain homotopy equivalence assigned to moving the basepoint past a crossing.}
\end{figure}

\subsection{Based knots and connected cobordisms}
\label{sec:decorated-cobord}

We now restrict ourselves to the case of based knots and discuss maps on their reduced homology coming from connected cobordisms between them. We will view cobordisms in $\R^3 \times [0,1]$ as movies where each frame is a slice $\R^3 \times \{t\}$ at time $t$.

Let $(K_i, p_i)$ for $i=0,1$ be based knots in $\R^3$ with respective diagrams $D_i$. Consider an oriented, connected cobordism $F \subset \R^3 \times [0,1]$ from $K_0$ to $K_1$ (i.e.\ $K_i = F \cap (\R^3 \times \{i\})$ for $i=0,1$). Choose an embedded arc $\Gamma$ from $p_0 \in K_0$ to $p_1 \in K_1$; this is possible because $F$ is connected. We call the pair $(F,\Gamma)$ a \emph{decorated cobordism} from $K_0$ to $K_1$.
We claim that it is always possible to isotope $F$, fixing its boundary, such that the following conditions are met. Let $\pi: \R^3 \times [0,1] \onto \R^2$ be the diagram projection map onto the first two coordinates of $\R^3$.
\begin{enumerate}[label= (DC-\arabic*)]
    \item The path $\Gamma$ is transverse to every slice $\R^3 \times \{t\}$ for $t \in [0,1]$. \label{item:Gamma-transverse}
    \item The height function $h_F$ on $F$ given by projection to the time coordinate is Morse, and the critical values (times when $h_F$ has a critical point) are isolated. \label{item:height-morse}
    \item For all but finitely many times $t$, the projection $\pi(F \cap (\R^3 \times \{t\}))$ is a link diagram (i.e.\ an immersion with isolated double points only). \label{item:projection-generic}
    \item For all but finitely many times $t$, the projection of the basepoint $\pi(\Gamma(t))$ is disjoint from the set of singular points of the projection under $\pi$ of $F \cap (\R^3 \times \{t\})$ to the plane $\R^2$, and from the critical points of $h_F$.  \label{item:good-basepoints}
\end{enumerate}
We briefly explain how to arrange this. 
Given any decorated connected cobordism $(F, \Gamma)$ from $K_0$ to $K_1$, we can perform an ambient isotopy that transforms $\Gamma$ into the ``straight" path from $p_0$ to $p_1$ (that is, the path is given by the point $p_0$ crossed with the interval). This is possible because $\Gamma$ is codimension 3 in $\R^3 \times [0,1]$. This achieves \ref{item:Gamma-transverse}. 
By a standard argument, \ref{item:height-morse} and \ref{item:projection-generic} can be achieved by perturbations. Since transversality is an open condition, our new $\Gamma$ may retain its transversality to the time slices, maintaining \ref{item:Gamma-transverse}.
To arrange \ref{item:good-basepoints}, observe that there are three types of ``diagrammatic" special points lying on the surface $F$: 
\begin{enumerate}
    \item the Morse singular points, corresponding to the isolated critical points of \ref{item:height-morse}, \label{item:special-morse}
    \item the Reidemeister singular points, corresponding to the isolated critical points of \ref{item:projection-generic}, \label{item:special-reid}
    \item and the diagram crossings, which comprise a 1-dimensional submanifold of $F$ traced out by the double points under the projection $\pi$. \label{item:special-crossings}
\end{enumerate}
By perturbing, we may arrange that the 0-dimensional special points of \eqref{item:special-morse} and \eqref{item:special-reid} are disjoint from the crossing arcs in \eqref{item:special-crossings}, as well as from each other. The union of all these special points is a 1-submanifold $SP$; since $\Gamma$ is also 1-dimensional, we finish by slightly perturbing $\Gamma$ in $F$ so that it is transverse $SP$ and the set of points where $\Gamma$ intersects $SP$ (points where the diagrammatic basepoint is passing through a crossing) is finite and disjoint from 0-dimensional components of $SP$, achieving \ref{item:good-basepoints}. 

From now on, we will assume all decorated cobordisms $(F, \Gamma)$ satisfy the above conditions; for clarity, we refer to them as \emph{nice decorated cobordisms}. 
We say $t \in (0,1)$ is a \emph{critical time} if at time $t$ we see a special point from \ref{item:height-morse}, \ref{item:projection-generic}, or an intersection of $\Gamma$ with $SP$.

We may thus diagrammatically present $(F,\Gamma)$ as a finite sequence of \emph{elementary cobordisms} between based link diagrams. These consist of  Morse moves away from the  basepoint, Reidemeister moves disjoint from the basepoint, and moving the basepoint past a crossing. See also   \cite[Section 2.3]{Baldwin-Hedden-Lobb}.

\begin{remark}
\label{rmk:R move involving basepoint}
A nice cobordism may a priori contain a segment representing a Reidemeister move involving the basepoint. However, we may always replace this segment of the movie with an equivalent one that first moves the basepoint outside of the region affected by the Reidemeister isotopy and then later moves it back in.
\end{remark}

\begin{definition}
\label{def:basepoint movie maps}
Given based link diagrams $(D, p)$ and $(D', p')$ which are related by one of the above elementary moves, we assign a chain map $\phi : \CKhred(D, p) \to \CKhred(D', p')$ as follows. For Morse moves and Reidemeister moves (both disjoint from the marked point), note that the usual chain map $\CKh(D) \to \CKh(D')$ preserves the respective reduced  complexes, and we define $\phi$ to be the restriction of this usual map. For a move which slides the basepoint past a crossing, the chain map is realized as a sequence of Reidemeister moves away from basepoints, followed by the global move in Figure \ref{fig:pt at infinity}, and finally followed by another sequence of Reidemeister moves (see Remark \ref{rem:no-funct}) away from basepoints. See Figure \ref{fig:around-town} for a schematic. 
\end{definition}

\begin{remark}
\label{rem:no-funct}
There are various choices made, for example for the path $\Gamma$ and Reidemeister moves before and after the global move in Definition \ref{def:basepoint movie maps}. Any choice is sufficient for our purposes; we make no claims about the functoriality of these chain maps and only require their existence. 
\end{remark}

Note that the maps of reduced complexes assigned to Reidemeister moves and basepoint-moving are chain homotopy equivalences.

\subsection{Localization and canonical Lee generators}

Consider the reduced, localized chain complex 
\[
\CKhredD(D,p) := \CKhred(D,p) \o_R R_\D.
\]
Recall the elements $e_1, e_2 \in R_\D$ from \eqref{eq:e_1 and e_2}. The following definition is a reduced version of Lee's canonical generators. 

\begin{definition}
\label{def:reduced Lee gens}
Let $(D,p)$ be a based link diagram. Given an orientation $o$ of $D$, consider the checkerboard coloring of the oriented resolution $D_o$ where the region to the left of $p$ is shaded. For a circle $C$ in $D_o$, label the circle $e_1$ if the region to the left of $C$ is shaded, and otherwise label it $e_2$. Denote this element $\r_o(D,p) \in \CKhredD(D,p)$; note that the marked circle in $D_o$ is labeled $e_1$. 
\end{definition}

The element $\r_o(D,p)$ is equal to either $\s_o$ or $\s_{\b{o}}$ (see Algorithm \ref{alg:ABcolors}). Note also that $\r_o(D,p) = \r_{o'}(D,p)$, if and only if $o=o'$ or $\b{o} = o'$, so if $D$ has $l$ components then there are $2^{l-1}$ such generators. Recall the algebra involution $\iota$ from \eqref{eq:iota involution}. The following is immediate. 

\begin{lemma}
\label{lem:reduced Lee generators}
For a based link diagram $(D,p)$ with $l$ components, $\KhredD(D,p)$ is a free $R_\D$-module of rank $2^{l-1}$, with basis given by $\{[\r_o(D,p)]\}$ over all orientations $o$ of $D$. Moreover, if $(D,p')$ is obtained from $(D,p')$ by moving the basepoint $p$ past a single crossing, then $\iota(\r_o(D,p)) =\r_o(D,p')$. 
\end{lemma}

Recall from Section \ref{sec:Inverting the discriminant} that the involution $I$ swaps $e_1$ and $e_2$ in the localized theory (up to a sign). Below, we assume all states (pure tensors) are written with the component containing the basepoint first. For an element $y\in \CKhredD(D,p)$, we write $e_1 \cdot y$ for the basepoint action that multiplies the based circle by $e_1$. Similarly, $e_2\cdot$ indicates the action where the tensor factor corresponding to the based component is multiplied by $e_2$.

\begin{remark}
The basepoint action $e_1 \cdot$ does not quite agree with the action induced on the reduced complex by $f_p$, which is multiplication by only the numerator $X-U$ of $e_1$. This distinction will not be important for our purposes, since $V-U$ is invertible in the localized theory.
\end{remark}

The  non-reduced localized complex splits as two copies of the reduced localized complex. 
Namely, we have a distinguished isomorphism
\[
 \CKhD(D) \cong \CKhredD(D,p) \oplus \CKhredD'(D,p),
\]
where the first summand is generated by generators where the marked component is labeled $e_1$, and the second summand $\CKhredD'(D,p)$ is generated by those where the marked component is labeled $e_2$. It follows from the formulas in Equation \eqref{eq:mult and comult on X-U, X-V} that both are indeed subcomplexes of $\CKhD(D)$, and moreover the involution $I$ gives an isomorphism $\CKhredD(D,p) \cong \CKhredD'(D,p)$. 
This splitting descends to a distinguished isomorphism 
\[
    \KhD(K) \cong \KhredD(K,p) \oplus \KhredD(K,p) = R_\D \langle [\r_o] \rangle \oplus R_\D \langle [\iota(\r_o)] \rangle.
\]
for a knot $K$.

\begin{lemma}
\label{lem:reduced Lee generator under basepoint moving}
Let $(D, p)$ be a based link diagram, and suppose $(D,p')$ is obtained from $(D,p)$ by moving the basepoint past a single crossing. Let $\phi \: \KhredD(D,p) \to \KhredD(D,p')$ denote the isomorphism from Definition \ref{def:basepoint movie maps}. Then for any orientation $o$ of $D$, we have $\phi ([\r_o(D,p)]) =
\pm [\r_o(D,p')]$. 
\end{lemma} 

\begin{proof}

Write $\phi$ as a composition $\KhredD(D,p) \rar{\phi_0} \KhredD(D_1, p_1) \rar{\phi_1} \KhredD(D_2, p_2) \rar{\phi_2} \KhredD(D, p')$, where $\phi_0$ is induced by the first sequence of Reidemeister moves, $\phi_1$ is the global move of Figure \ref{fig:pt at infinity}, and $\phi_2$ is another sequence of Reidemeister moves. Let $o_1$ and $o_2$ denote the orientations on $D_1$ and $D_2$. Without loss of generality, assume $\r_o(D,p) = \s_o(D)$. By \cite[Proposition 2.13]{Sano-divisibility}, $\phi_0([\s_o(D)]) = \pm (U-V)^l [\s_{o_1}(D_1)]$  for some $l$. We have $\phi_1([\s_{o_1}(D_1)]) = [\s_{\b{o}_2}(D_2)]$ because the checkerboard coloring reverses. Applying \cite[Proposition 2.13]{Sano-divisibility} again we obtain $\phi_2([\s_{\b{o}_2}(D_2)]) = \pm(U-V)^{l'} [\s_{\b{o}}(D)]$ for some $l'\in \Z$. By the last statement in Lemma \ref{lem:reduced Lee generators}, $\r_{o}(D,p') = \iota(\r_o(D,p)) =  \s_{\b{o}}$, so we have $\phi ([\r_o(D,p)]) =
\pm (U-V)^{l+l'} [\r_o(D,p')]$. Finally, $l+l'=0$ for quantum grading reasons. 
\end{proof}

In fact, any decorated connected cobordism between knots gives an isomorphism on the localized reduced homology. 

\begin{proposition}
\label{prop:reduced Lee generators under connected cob}
Let $(F,\Gamma)$ be a decorated connected cobordism from $(K, o)$ to $(K', o')$. 
Pick based diagrams $(D, p)$ and $(D', p')$ for $(K, p)$ and $(K', p')$ respectively.
There exists a chain map $\phi \: \CKhred(D, p) \to \CKhred(D', p')$ of degree $\chi(F)$ such that, after inverting $\D$, the induced map on homology is an isomorphism. 
\end{proposition}

\begin{proof}
The argument is similar to \cite[Proposition 4.1]{Rasmussen} and \cite[Proposition 3.17]{Sano-divisibility}.  Represent $(F,\Gamma)$ as a movie of based link diagrams $(D,p) = (D_0, p_0), (D_1,p_1),\ldots, (D_k, p_k) = (D',p')$, where $(D_i, p_i) = F\cap (\R^3\times \{i/k\})$, and $(D_{i+1}, p_{i+1})$ is obtained from $(D_i, p_i)$ by one of elementary moves as described in Definition \ref{def:basepoint movie maps}. For each elementary move $(D_i, p_i) \to (D_{i+1}, p_{i+1})$ consider the map 
\[
\phi_i : \Khred(D_i, p_i) \to \Khred(D_{i+1}, p_{i+1})
\]
as defined in Definition \ref{def:basepoint movie maps}. Let $\phi = \phi_k \circ \cdots \circ \phi_0$ denote their composition and let $\psi_i = \phi_{i-1} \circ \cdots \circ \phi_0$ denote their composition until the $i$-th slice. We will show that $\phi([\r_o(D, p)] = \pm (V-U)^k \r_{o'}(D',p')$ for some $k\in \Z$.

For $1\leq i \leq k$, let $F_i = F\cap (\R^2\times [0,i/k])$. An orientation of $F_i$ is \emph{permissible} if it induces the orientation $o$ on $D$. Let $O_i$ denote the set of permissible orientations on $F_i$. An  orientation of $F_i$ induces an orientation on $D_i$, and moreover since $F$ is connected, distinct orientations on $F_i$ induce distinct orientations on $D_i$. For $\alpha \in O_i$, we also write $\alpha$ to denote the induced orientation on $D_i$. The statement follows from three claims: \\
\\
\noindent
\textbf{Claim 1:} For any $\alpha \in O_i$, $\displaystyle \phi_i([\r_\alpha(D_i,p_i)]) = \sum_{\beta \in \Omega_{i+1}} \pm (V-U)^{l_\beta} [\r_{\beta}(D_{i+1},p_{i+1})]$, where  $\Omega_{i+1} \subset O_{i+1}$ is the set of orientations $\beta$ that extend $\alpha$. This set may be empty. \\
\\
\textbf{Claim 2:} For any $y\in \Khred(D_i, p_i)$ of the form $\displaystyle y= \sum_{\alpha \in \Omega_i} \pm (V-U)^{l_\alpha} [\r_{\alpha}(D_i,p_i)]$ where $\Omega_i \subset O_i$, we have that $ \displaystyle \phi_i(y) = \sum_{\beta \in \Omega_{i+1}} \pm (V-U)^{l_\beta} [\r_{\beta}(D_{i+1},p_{i+1})]$ for some $\Omega_{i+1} \subset O_{i+1}$. \\
\\
\textbf{Claim 3:} For each $0\leq i \leq k-1$, $ \displaystyle \psi_i ([\r_o(D,p)]) = \sum_{\alpha \in \Omega_i} \pm (V-U)^{l_\alpha} [\r_{\alpha}(D_i,p_i)]$, where $\Omega_i \subset O_i$ is nonempty. 

Before verifying the above claims, note that Claim 3 applied to $\psi_{k-1} = \phi$ implies the statement of the lemma, since there is only one permissible orientation on $F_1 = F$ (namely, the given orientation on $F$ which induces $o'$ on $D'$).

To show Claim 1, suppose first that $(D_i, p_i) \to (D_{i+1}, p_{i+1})$ is a Morse or Reidemeister move away from the basepoint. Note that $\r_\alpha(D_i,p_i)$ is equal to one of $\s_\alpha(D_i)$ or $\s_{\b{\alpha}}(D_i)$. Without loss of generality, suppose $\r_\alpha(D_i,p_i) = \s_\alpha(D_i)$. By \cite[Proposition 3.17]{Sano-divisibility}, 
\[
\phi_i([\s_\alpha(D_i)]) = \sum_{\beta  \in \Omega_{i+1}} \pm (V-U)^{l_\beta} [\s_{\beta}(D_{i+1})].
\]
Moreover, since in the checkerboard coloring used to define the reduced Lee generators, the unbounded region is colored the same in $D_i$ and $D_{i+1}$, we also have $\s_{\beta}(D_{i+1}) = \r_{\beta}(D_{i+1}, p_{i+1})$, which verifies Claim 1 in this case. If the move $(D_i, p_i) \to (D_{i+1}, p_{i+1})$ slides the basepoint past a crossing, then Claim 1 follows from Lemma \ref{lem:reduced Lee generator under basepoint moving}. 

Claim 2 follows from Claim 1, since distinct permissible orientations on $F_i$ extend to distinct permissible orientations on $F_{i+1}$. 

Claim 3 follows from repeatedly applying Claim 2 and the fact that a generator $\r_\alpha(D_i,p_i)$ is sent to zero by $\phi_i$ only if two oriented strands are incompatibly merged, which can never appear in the oriented cobordism $F$.  
\end{proof}

For a based link diagram $(D,p)$, let $T(D,p) \subset \Khred(D,p)$ denote the $R$-torsion submodule.

\begin{corollary}
\label{cor:reduced induced map is injective mod torsion}
Let $(K,p)$ and $(K',p')$ be oriented knots and $(F,\Gamma)$ a decorated connected cobordism from between them. Fix oriented based diagrams $(D,p)$ and $(D',p')$ for $(K,p)$ and $(K',p')$ respectively. Then the  map
\[
\phi_F \: \Khred(D,p)/T(D,p) \to \Khred(D',p')/T(D',p')
\]
from Proposition \ref{prop:reduced Lee generators under connected cob} is injective.
\end{corollary}
\begin{proof}
The proof is similar to that of Corollary \ref{cor:induced map is injective mod torsion}.
\end{proof}

\subsection{Concordance invariants from the reduced  theory}

In this section, we consider based knots $K_0$ and $K_1$ and a decorated cobordism $(F,\Gamma)\: (K_0, p_0) \to (K_1, p_1)$ between them. We assume that $(F, \Gamma)$ is already in the preferred position described in Section \ref{sec:decorated-cobord}.

If $(K,p)$ is a based knot with orientation $o$, the localized reduced homology $\KhredD(D,p)$ is a free rank $1$ $R_\D$-module, generated by $\r_o(D,p)$. 
Let $\grtred$ denote the grading induced by $\gr_t$ on $\CKhred(D)  \xhookrightarrow{i} \CKh(D)$; note that the inclusion map $i$ drops the quantum grading by 1, since we have defined $\CKhred(D,p) = \im (f_p) \{1\}$. We again abuse notation and also use $\grtred$ to denote the induced gradings on $\CKhredD(D,p)$ as well. Note that due to the quantum grading shift, $\grtred(x) = \grt(ix) + 1$.

\begin{definition}
Define a reduced version of the concordance invariant $s_t$ as follows:
\[
    \til{s}_t(K,p) = \max \{  \grtred(\b{\eta}) \st 
    \b{\eta} \in \Khred(D,p) / T(D,p) \text{ is nonzero}        
    \}
\]
for any based diagram $(D,p)$ of $(K,p)$. 
Note that there is no additional quantum grading shift: $\til{s}_t(K,p) = \grtred(D)$. 
\end{definition}

Just from the definitions, we see that 
\[
    \tilde s_t(K) = \grtred(D) \leq \grt(D) + 1 = s_t(K) + 2.
\]
The inequality arises because a $\grtred$-maximizing cycle $y \in \CKhred(D,p)$ might not maximize $\gr_t$ after inclusion into $\CKh(D)$: $\grtred(y) = \grt(i y)+1 \leq \grt(D) + 1$. 

\begin{theorem}
\label{thm:reduced-concordant}
If based knots $(K,p)$ and $(K,p')$ are concordant, then $\til{s}_t(K,p) = \til{s}_t(K',p')$.
\end{theorem}

\begin{proof}
The proof is similar to that of Theorem \ref{thm:concordance-invariant}, using Corollary  \ref{cor:reduced induced map is injective mod torsion} in place of Corollary \ref{cor:induced map is injective mod torsion}.
\end{proof}

\subsection{Mirrors}
\label{sec:mirrors}

We now address the behavior of $s_t$ under mirroring. While we expect $s_t$ to be a concordance homomorphism, i.e.\ a homomorphism from the knot concordance group to the group of piecewise-linear functions on the interval $[0,2]$, the following is currently unknown:\
\begin{question}
For knots $K$ and $J$, is $s_t(K \# J) = s_t(K) + s_t(J)$ for all $t \in [0,2]$?
\end{question}
If the answer to the above is yes, then because $s_t=0$ for the unknot, we would know that for any knot $K$, $s(-K) = -s(K)$, where $-K$ is the mirror of the knot $K$.

In the meantime, we can still study the structure of the $\grt$-filtered complex for the mirrored knot.

\subsubsection{$\F$ is self-dual}

Let $S$ be a commutative ring and $\G = (B, \iota, \epsilon, m, \Delta)$ a Frobenius extension of $S$ with underlying $S$-module $B$, trace $\epsilon \: B\to S$, unit $\iota \: S\to B$, multiplication $m \: B\o B \to B$, and comultiplication $\Delta \: B\to B\o B$. Recall from \cite{Khfrobext} the \emph{dual} Frobenius extension $\G^* = (B^*, \epsilon^*, \iota^*, \Delta^*, m^*)$ with underlying $S$-module the  dual $B^* = \Hom_S(B,S)$, and structure maps given by dualizing the structure maps of $\G$. Following \cite{Khfrobext},  $\G$ is said to be \emph{self dual} if $\G \cong \G^*$ as Frobenius systems; that is, there exists an $S$-linear isomorphism $\phi \: B\to B^*$ which identifies $(\iota, \epsilon, m, \Delta)$ with $(\epsilon^*, \iota^*, \Delta^*, m^*)$. 

\begin{proposition}
\label{prop:self dual}
The Frobenius system $\F =(A,\iota, \varepsilon, m, \Delta)$ defined in Section \ref{sec:link homology} is self dual.
\end{proposition}

\begin{proof}
We will define an auxiliary Frobenius system $\F'$ and show that $\F \cong \F' \cong \F^*$. Let $1^*, X^* \in A^*$ denote the dual vectors to $1$ and $X$, so that $1^*(y) = \delta_{1,y}$, $X^*(y) = \delta_{X,y}$, for $y\in \{1, X\}$. Structure maps for $\F^*$ are given below. 
\begin{align}
\begin{aligned}
    \iota^* \: A^* & \to R   & m^* \: A^*  & \to A^* \o A^*  \\
    1^* & \mapsto 1  & 1^*  &  \mapsto 1^* \o 1^* - UV X^* \o X^* \\
    X^* & \mapsto 0   & X^* & \mapsto 1^*\o X^* + X^* \o 1^* + (U+V) X^* \o X^* 
    \end{aligned}
\end{align}
\vskip1em
\begin{align}
\begin{aligned}
  \varepsilon^* \: R & \to A^* & \Delta^* \: A^* \o A^* & \to A^*  \\
   1 & \mapsto X^* & 1^* \o 1^* & \mapsto -(U+V) 1^* - UV X^* \\
   & & X^* \o 1^* , 1^* \o X^* & \mapsto   1^*  \\
   & & X^* \o X^* & \mapsto X^*
    \end{aligned}
\end{align}
Define a Frobenius system $\F'$ whose underlying $R$-algebra is \[
A' = R[X]/((X+U)(X+V)),
\]
unit the natural inclusion $R\hookrightarrow A'$, trace $\varepsilon'(1) = 0, \varepsilon'(X)=1$, and comultiplication 
\[
\Delta'(1) = (X+U)\o 1 + 1\o (X+V),\ \ \ \Delta'(X) = X\o X - UV 1\o 1.
\]
Note that $\F'$ is just $\F$ with $U$ and  $V$ replaced with $-U$ and $-V$, respectively. From the above structure formulas, we see that\footnote{This was observed in \cite[Example 2]{Khfrobext}.} $\F^* \cong \F'$ via $1^* \mapsto X$, $X^* \mapsto 1$. 

Next, consider the $R$-algebra isomorphism $\varphi \: A \to A'$ given by $\varphi(1) = 1$, $\varphi(X) = X+ U + V$. We will show that $\varphi$ yields an isomorphism $\F' \cong \F$. The unit and co-unit identities are clear. For comultiplication, we have
\begin{align*}
    \Delta'\circ \varphi (1)    &= 1\o X + X\o 1 + (U+V)1\o1 \\  
    &= 1\o (X+U+V) + (X+U+V) \o 1 - (U+V) 1\o 1 \\
     &= (\varphi\o \varphi) \circ \Delta(1)   \\
    \Delta' \circ \varphi(X) &= X\o X - UV 1\o 1 + (U+V) \left[1\o X+ X\o 1 + (U+V) 1\o 1 \right]\\
    &=X \o X + (U+V) \o X + X \o (U+V) + (U+V)^2 1\o 1 - UV 1\o 1 \\
    &= (X+U+V)\o (X+U+V) -UV 1\o 1 \\
    &= (\varphi \o \varphi) \o \Delta(X) 
\end{align*}
which completes the proof. 
\end{proof}

\begin{remark}
The proof of Proposition \ref{prop:self dual} also shows that $(R^{\rm sym},A^{\rm sym})$ as defined in Remark \ref{rem:U(2) pair} is self-dual via $1\mapsto X^*$, $X\mapsto 1^* + E_1 X^*$. 
\end{remark}

\subsubsection{Graded duals}

We pause to recall some notions regarding gradings. Given a graded $R$-module $M$, let $M^* = \Hom_R(M,R)$ denote its dual and $M^*_{\gr} \subset M^*$ the graded dual, 
\[
M^*_{\gr} := \bigoplus_{i\in\Z} \Hom_{i}(M,R),
\]
where $\Hom_i(M,R)$ is the abelian group of all $R$-linear maps $M\to R$ of degree $i$. Note that $M^*_{\gr}$ is a graded $R$-module and that $M^*_{\gr} =M^*$ when $M$ is finitely generated. Moreover, if $M$ is free with homogeneous basis $\{m_1, \ldots, m_k\}$, then the dual elements $\{m^*_1, \ldots, m^*_k\}$ form a homogeneous basis for $M^*_{\gr} = M^*$, with $\deg(m_i^*) = - \deg(m_i)$. Chain groups in Khovanov homology are finitely generated and free over the ground ring, so we will omit the subscript $\gr$ in the graded dual. 

Denote by $\gamma \: A\to A^*$ the isomorphism constructed in the proof of Proposition \ref{prop:self dual}, given by $\gamma(1) = X^*$, $\gamma(X) = 1^* + (U+ V)X^*$. 
Note that the isomorphism $\gamma \: A\to A^*$ is grading preserving.

\subsubsection{Mirror complex}

Given a link diagram $D$, denote by $\b{D}$ its mirror image diagram; if $D$ is based at $p$, let $\b{p}$ be the image of $p$ in $\b{D}$. 
Let $\CKh(D)^!$ be the chain complex given by
$\left(\CKh^{-i}(D)\right)^*$ in homological grading $i$ and whose differential is obtained by dualizing the differential in $\CKh(D)$. Using well-known arguments \cite[Section 7.3]{Kh1}, Proposition \ref{prop:self dual} implies the following. We recall the argument since it will be used in Proposition \ref{prop:reduced mirror}.

\begin{corollary}
\label{cor:dual of mirror iso}
There is an isomorphism $\CKh(D) \cong  \CKh(\b{D})^!$ of chain complexes of graded $R$-modules.
\end{corollary} 

\begin{proof}
The isomorphism is defined on each vertex in the cube of resolutions. Let $n$ denote the number of crossings of $D$, $n_+$ of which are positive and $n_-$ of which are negative. Let $\b{n}_+ = n_-$, $\b{n}_- = n_+$ denote the number of positive and negative crossings in $\b{D}$. For  $u\in \{0,1\}^n$, let $\b{u} \in \{0,1\}^n$ denote the $n$-tuple obtained by swapping all $0$'s to $1$'s and all $1$'s to $0$'s. We have $D_u = \b{D}_{\b{u}}$.  The smoothing $D_u$ contributes to the homological grading $\lr{u} - n_-$ part of $\CKh(D)$, while $\b{D}_{\b{u}}$ contributes to the homological grading $\lr{\b{u}} - \b{n}_-  = - (\lr{u} - n_-)$ part of $\CKh(\b{D})$. Letting $k_u$ denote the number of circles in $D_u$, the isomorphism $\CKh(D) \cong \CKh(\b{D})^!$ is obtained by applying the isomorphism
\[
\gamma_u \: A^{\o k_u} \to (A^*)^{\o k_u}
\]
induced by $A \cong A^*$ at each smoothing. Proposition \ref{prop:self dual} implies that the $\gamma_u$ assemble into a chain map. Finally, one can check that $\gamma_u$ is grading preserving after accounting for quantum grading shifts.
\end{proof}

\begin{proposition}
\label{prop:reduced mirror}
There is an isomorphism $\CKhred(D,p) \cong \CKhred(\b{D},\b{p})^!$ of chain complexes of graded $R$-modules. 
\end{proposition}

\begin{proof}
Fix a resolution $D_u$ of $D$, and order its circles $Z_1, \ldots, Z_{k_u}$, with $Z_1$ the marked circle. Then $D_u$ contributes the direct summand $R(X-U) \o A^{\o k_u-1}$ to $\CKhred(D,p)$, while $\b{D}_{\b{u}}$ contributes the direct summand $\left(R(X-U)\right)^* \o \left(A^* \right)^{\o k_u-1}$ to  $\CKhred(\b{D},\b{p})^!$. Consider the $R$-linear isomorphism 
\[
\til{\gamma}_u \: R(X-U) \o A^{\o k_u-1} \to \left(R(X-U)\right)^* \o \left(A^* \right)^{\o k_u-1}
\]
given by $(X-U) \o y_1\o \cdots \o y_{k_u-1} \mapsto (X-U)^* \o \gamma(y_1) \o \cdots \o \gamma(y_{k_u-1})$. 
It is straightforward to verify that $\til{\gamma}_u$ is grading preserving after the appropriate grading shifts are applied. We let $R_U = (X-U)\cdot A$ be the free $R$-module with basis $X-U$. 

We check that $\til{\gamma}_u$ commutes with each saddle map. When the saddle involves unmarked circles, commutativity follows from Proposition \ref{prop:self dual}. It suffices then to check commutativity of the two diagrams in \eqref{eq:two cases}; the left (\resp right) diagram corresponds to a circle merging  with (\resp splitting from) a marked circle. 

   \begin{equation}
    \label{eq:two cases}
         \begin{tikzcd}[column sep = small]
        R_U \o A \ar[r] \ar[d, "\til{\gamma}_u"'] & R_U \ar[d, "\til{\gamma}_u"] \\
         \left(R_U\right)^* \o A^* \ar[r] & \left(R_U\right)^*
    \end{tikzcd}\hskip2em
    \begin{tikzcd}[column sep = small]
    R_U \ar[d, "\til{\gamma}_u"'] \ar[r] & R_U \o A \ar[d, "\til{\gamma}_u"] \\
    \left(R_U\right)^* \ar[r] & \left(R_U\right)^*\o A^*
    \end{tikzcd}
    \end{equation}
The bottom horizontal maps are obtained by dualizing comultiplication and multiplication in the reduced complex when at least one of the circles involved is marked:
\begin{align*}
    \left(R_U\right)^* \o A^* & \to \left(R_U\right)^* & \left(R_U\right)^* & \to \left(R_U\right)^* \o A^* \\
    (X-U)^*\o 1^*& \mapsto  -U (X-U)^* & (X-U)^* &\mapsto (X-U)^* \o (1^* + V X^*).\\
     (X-U)^*\o X^*& \mapsto   (X-U)^* &  &
\end{align*}
Commutativity of the diagrams in \eqref{eq:two cases} follows easily.  
\end{proof}

\begin{remark}
Note that Proposition \ref{prop:reduced mirror} does \emph{not} imply a particular relationship between $\tilde s_t(K)$ and $\tilde s_t(\b{K})$. This is because our ground ring $R = \bF[U,V]$ is not a principal ideal domain, and thus the dualizing functor $\Hom_R(\cdot, R)$ is not exact. Nonvanishing higher $\Ext$ groups prevent us from identifying the homologies of $\CKhred(D,p)$ and $\CKhred(\b{D},\b{p})$ using simple grading reversals, as is the case with Khovanov homologies over PIDs.  
\end{remark}

\bibliographystyle{alpha}
\bibliography{main}

\end{document}